\documentclass[reqno,12pt]{amsart}
\usepackage{verbatim}
\usepackage{color}
\usepackage{amsthm,amsmath,amsfonts,amssymb,amscd,mathrsfs,graphics}
\usepackage{times,float}
\usepackage{appendix}
\usepackage{txfonts}
\usepackage{supertabular}
\usepackage[colorlinks,linkcolor=black,anchorcolor=blue,citecolor=green]{hyperref}

\usepackage{ifpdf}
\usepackage{tikz}
\usepackage{enumerate}
\usepackage{pst-tree}
\usepackage{fancyhdr}

\allowdisplaybreaks

\usetikzlibrary{arrows,calc}
\tikzset{
>=stealth',
help lines/.style={dashed, thick},
axis/.style={<->},
important line/.style={thick},
connection/.style={thick, dotted},
}
\newcommand{\LD}{\langle}
\newcommand{\RD}{\rangle}
\newcommand{\CC}{\mathbb{C}}
\newcommand{\cal}{\mathcal}
\newcommand{\be}{\mathbf{e}}

\newtheorem{thm}{Theorem}[section]
\newtheorem{lm}[thm]{Lemma}

\newtheorem{conj}[thm]{Conjecture}

\theoremstyle{definition}
\newtheorem{rem}[thm]{Remark}

\newtheorem{df}[thm]{Definition}

\theoremstyle{remark}

\title[LG/CY Correspondence of all genera for elliptic orbifold $\mathbb{P}^1$]{Landau-Ginzburg/Calabi-Yau Correspondence of all genera for elliptic orbifold $\mathbb{P}^1$}

\author{Marc Krawitz}
\address{Chicago, IL}
\email{marc.krawitz@gmail.com}

\author{Yefeng Shen}
\address{Department of Mathematics, University of Michigan, Ann Arbor, MI}
\email{yfschen@umich.edu}

\begin{document}
\setlength{\unitlength}{1mm}

\maketitle

\tableofcontents

\section{Introduction}
    There is a great deal of interest recently in studying the so called Landau-Ginzburg /Calabi-Yau correspondence or LG/CY correspondence, a famous duality from phys-ics. Mathematically, the LG/CY correspondence concerns about the equivalence of two mathematical theories originating from a quasihomogeneous polynomial of Calabi-Yau type. A polynomial $W: \mathbb{C}^N\rightarrow \mathbb{C}$ is quasihomogenous if
    there is an $N$-tuple of rational numbers (weights or charges) $(q_1, \cdots, q_N)$ such that for any $\lambda\in \mathbb{C}^*$,
    $$W(\lambda^{q_1}z_1\cdots, \lambda^{q_N}z_N)=\lambda W(z_1, \cdots, z_N).$$
    We assume that $W$ is nondegenerate in the sense that it defines an isolated singularity at the origin. $W$ is called of {\em Calabi-Yau type}
    if $\sum_i q_i=1$.  Another piece of data is a subgroup $G$ of the diagonal symmetry group $G_{\textrm{max}}$, where
    $$G_{\textrm{max}}:=\{\textrm{Diag}(\lambda_1, \cdots, \lambda_N); W(\lambda_1 z_1, \cdots, \lambda_N z_N)=W(z_1, \cdots, z_N)\}$$
    $G_{\textrm{max}}$ always contain a special element $J=\textrm{Diag}(e^{2\pi i q_1},\cdots, e^{2\pi i q_N}).$  $G\subset G_{\textrm{max}}$ is {\em admissible} if it contains $J$.

    The geometric implication of a Calabi-Yau type quasihomogeneous polynomail is that $W=0$ defines a Calabi-Yau hypersurface $X_W$ in the weighted projective space
    $WP(c_1, \cdots, c_N)$, where $q_i=c_i/d$ for common denominator $d$. $G$ acts naturally on $X_W$ with the kernel $\langle J \rangle$.
    Hence $\widetilde{G}:=G/\LD\,J\RD$ acts faithfully on $X_W$. One side of the LG/CY-correspondence is the Gromov-Witten theory of the quotient of the Calabi-Yau hypersurface, $\{X_W=\{W=0\}\}/\widetilde{G}$.
    Gromov-Witten theory is now well-known in mathematics. Its main elements are
    \begin{itemize}
    \item[(1)] a state space: Chen-Ruan orbifold cohomology $H^*_{CR}(X_W/\widetilde{G})$;
    \item[(2)] numerical
    invariants $\langle \tau_{l_1}(\alpha_1), \cdots, \tau_{l_n}(\alpha_n)\rangle_{g,n,\beta}$ defined by  a virtual counting of stable maps.
    Here, $g$ is the genus and $\beta$ is the fundamental class of stable maps. One often assembles them into a generating function
    $\cal F^{GW}_g(q)$ in infinitely many variable indexed by a basis $\{\alpha_i\}$ of the state space, a $z$-variable to keep track of the integer $l$ and a quantum or K\"{a}hler variable
    $q$ to keep track of $\beta$. One can further sum over genera to define the total ancestor potential function
    $$\cal A_{GW}(q)=\sum_g \hbar^{2g-2}{\cal F}^{GW}_g(q).$$
    We should emphasis that ${\cal F}^{GW}_g$ is only a formal power series;
    \item[(3)] $\langle \tau_{l_1}(\alpha_1), \cdots, \tau_{l_n}(\alpha_n)\rangle_{g,n,\beta}$ satisfies a set of axioms referred as {\em cohomological field theory axioms} (see \cite{ChenR} for details).
        \end{itemize}

    The other side of the LG/CY-correspondence is the FJRW-theory of the singularity $(W, G)$ constructed by Fan-Jarvis-Ruan \cite{FJR2} based on a proposal
    of Witten. FJRW-theory is very different from Gromov-Witten theory. However, it shares the same general structure with Gromov-Witten theory
    as an example of cohomological field theory (see the section 3 for details). For example, it has (1) a state space ${\cal H}_{FJRW}$; (2) numerical
    invariants by a virtual counting of solutions of the Witten equation, its generating functions ${\cal F}^{FJRW}_g(s), {\cal A}_{FJRW}(s)$
    ($s$ is a certain degree 2 variable playing the role of K\"{a}hler parameter);
    (3) satisfies the cohomological field theory axioms.

    Motivated by physics, Yongbin Ruan has formulated a striking mathematical conjecture to equate the two theories \cite{R}. The main goal of this paper
    is to prove this conjecture for elliptic orbifold $\mathbb{P}^1$. Ruan's conjecture is stated as follows.

   \begin{conj}
   \begin{itemize}
   \item[(1)] There is a graded vector space isomorphism ${\cal H}_{FJRW}\rightarrow H^*_{CR}(X_W/\widetilde{G}).$ Hence, we can identify the two state spaces.
   \item[(2)] There is a degree-preserving $\CC[z, z^{-1}]$-valued linear symplectic isomorphism
       ${\mathbb U}$ of so-called Givental symplectic vector spaces and a choice of analytic
       continuation of Givental cones ${\cal L}_{FJRW}$ and ${\cal L}_{GW}$ with respect to the K\"{a}hler parameter
    such that ${\mathbb U}_{\rm LG/CY} ({\cal L}_{FJRW})={\cal L}_{GW}.$
    \item[(3)] Up to an overall constant, the total potential functions up to a choice of
    analytic continuation are related by
    quantization of ${\mathbb U}_{\rm LG/CY} $; {i.e.}
    $${\cal A}_{GW}=\widehat{{\mathbb U}_{\rm LG/CY}}({\cal A}_{FJRW}).$$
    \end{itemize}
    \end{conj}

    Part (1) is called the {\em Cohomological LG/CY correspondence}, which has been verified in full generality by Chiodo-Ruan \cite{CR2}.
    Part (2) is the genus-0 LG/CY correspondence, which has been verified by Chiodo-Ruan \cite{CR1} for the quintic 3-fold, and by Chiodo-Iritani-Ruan
    \cite{CIR} for all Fermat hypersurfaces and $G=\langle J\rangle$. So far, not a single example is known for all genera. In this article,
    we will prove the above conjecture for three classes of orbifold $\mathbb{P}^1$. Namely, we consider three cubic polynomials with their maximal diagonal symmetry group.
    \begin{itemize}
    \item []$P_8:=x^3+y^3+z^3$,
    \item []$X^T_9:=x^3+xy^2+yz^2$,
    \item []$J^T_{10}:=x^3+y^3+yz^2$.
    \end{itemize}
    In Saito's convention for simple elliptic singularities, $P_8$ is also denoted by $\widetilde{E}_6$.
    It is easy to check that $X_{P_8}/\widetilde{G}_{\textrm{max}}=\mathbb{P}^1_{3,3,3}$, the orbifold $\mathbb{P}^1$ with three orbifold points of weights ($3,3,3)$; $X_{X^T_9}/\widetilde{G}_{\textrm{max}}=\mathbb{P}^1_{4,4,2}$ and $X_{J^T_{10}}/\widetilde{G}_{\textrm{max}}=\mathbb{P}^1_{6,3,2}$.
    They are quotients of elliptic curves. We will refer them as {\em elliptic orbifold $\mathbb{P}^1$}. For
    Ruan's conjecture to make sense, we need the generating function to be analytic with respect to the K\"{a}hler parameter. This is often a difficult problem in Gromov-Witten theory and interesting in its own right. Our first theorem (Theorem \ref{thm:convergent}) establishes it for both Gromov-Witten theory and FJRW-theory.
    It is convenient to consider the {\em ancestor correlator function} $\LD\LD \tau_{l_1}(\alpha_1), \cdots, \tau_{l_n}(\alpha_n)\RD\RD^{GW}_{g,n}({\bf t})$. (see the precise definition in Section 4)
    Among all the cases we consider, the state space is decomposed into $H_{<2}\oplus H_2$ where $H_2$ is a one-dimensional space of degree 2 classes and $H_{<2}$ is the  subspace of degree $<2$. Let ${\bf t}=(t, s)$. We can convert $s$ to the familiar $q$ variable by the substitution $q=e^s$.
    We  define $\LD\LD \tau_{l_1}(\alpha_1), \cdots, \tau_{l_n}(\alpha_n)\RD\RD^{FJRW}_{g,n}(\bf t)$ in the same way. The main difference is the absence of the $\beta$ variable.

    \begin{thm}\label{thm:convergent}

    \begin{itemize}
    \item[(1)] For the above three classes of elliptic orbifold $\mathbb{P}^1$'s,
    $$\LD\LD \tau_{l_1}(\alpha_1), \cdots, \tau_{l_n}(\alpha_n)\RD\RD^{GW}_{g,n}(t,s)$$
     converges to an analytic function near $t=0, Re(s)\ll0$ or $q=0$.
    \item[(2)]  For its FJRW counterparts,
    $$\LD\LD \tau_{l_1}(\alpha_1), \cdots, \tau_{l_n}(\alpha_n)\RD\RD^{FJRW}_{g,n}(t,s)$$
    converges to an analytic function near $t=0, s=0$.
    \end{itemize}
    (1) is often referred as the large volume limit while (2) can be referred to as the small volume limit.
    \end{thm}

    The main theorem of the article is

    \begin{thm}
    Ruan's conjecture for all genera holds for the above three classes of elliptic orbifold $\mathbb{P}^1$'s and their FJRW counterparts.
    \end{thm}

    Chiodo-Ruan \cite{CR3} has reframed Ruan's conjecture in the language of {\em global mirror symmetry}. We will take their approach. The global mirror symmetry of our
    examples involves B-model objects consisting of three families of elliptic singularities
\begin{eqnarray}
\begin{split}
&P_8(\sigma): x^3+y^3+z^3+\sigma xyz, &\sigma^3+27\neq0.\\
&X_9(\sigma): x^2z+xy^3+z^2+\sigma xyz, &\sigma^3+27\neq0. \\
&J_{10}(\sigma): x^3z+y^3+z^2+\sigma xyz, &\sigma^3+27\neq0.
\end{split}
\end{eqnarray}
We abbreviate $P_8=P_{8}(0),X_9=X_9(0),J_{10}=J_{10}(0)$. There is a B-model theory on the local miniversal deformation of singularities constructed
by Saito-Givental. In a companion article by Milanov-Ruan \cite{MR}, they have worked out a {\em global} Saito-Givental theory in the sense of allowing the parameter $\sigma$ to vary from $\sigma=0$ to $\sigma=\infty$. In particular, they show that Saito-Givental theory at $\sigma=0$ is related to Saito-Givental theory at $\sigma=\infty$ by an analytic continuation and quantization of symplectic transformations.
By \cite{CR3}, Ruan's conjecture will follow from the theorem of Milanov-Ruan and two mirror symmetry theorems. The main technical content of our papers is to prove the following {\em LG-to-LG mirror theorem} and {\em LG-to-CY mirror theorem}.

\begin{thm}\label{thm:LG-LG mirror}
When $W:=P_{8},X_{9},J_{10}$, we can choose the coordinates appropriately, such that
\begin{equation}
\cal{A}_{FJRW}^{W^T,G_{\textrm{max}}}=\cal{A}_{W}^{SG}.
\end{equation}
\end{thm}

\begin{thm}\label{thm:LG-to-CY mirror}
For each of the three types of elliptic orbifold $\mathbb{P}^1$, we can choose a coordinate system such that
\begin{equation}
\cal{A}_{GW}(X_{W^T}/\widetilde{G}_{\textrm{max}})=\cal{A}_{W_{\infty}}^{SG}.
\end{equation}
\end{thm}

\subsection{Relation to the work of Milanov-Ruan}

  There is an companion article of Milanov-Ruan \cite{MR}, and the two papers can be considered as a single entity. Basically, we are working on the A-model side of global mirror symmetry while they work on the B-side. The final conclusion on analytic continuation and quantization of symplectic transformation is drawn from their work on the B-side. On the other hand, they also draw our work at several critical places.
For example, the Saito-Givental generating function is only defined at the semi-simple loci. The analytic continuation is carried out precisely at non-semisimple loci. Milanov-Ruan's theorem is based on the assumption that the Saito-Givental generating function can be extended to non-semisimple loci. As far as we know, this is a very difficult problem. In this case, our convergence theorem and two mirror theorems provide the main ingredient of the proof of the extendability of the Saito-Givental generating function (see Lemma 3.2 in \cite{MR}). Another beautiful
theorem of Milanov-Ruan is the quasi-modularity of the Saito-Givental generating function. Using our LG-to-CY mirror symmetry theorem, they conclude of the quasi-modularity of Gromov-Witten theory of the corresponding elliptic orbifold $\mathbb{P}^1$.

The paper is organized as follows. In section two, we will review the B-model results of Milanov-Ruan. In section three, we will prove the convergence of FJRW-theory and the LG-to-LG mirror theorem. The convergence of Gromov-Witten theory and the LG-to-CY mirror theorem will be established in section four.

\subsection{Acknowledgements}

   The genus-0 Gromov-Witten theory of the above elliptic orbifold $\mathbb{P}^1$ has been studied extensively by Satake and Takahashi \cite{SatT}. Among other things, they established the genus-0 LG-to-CY mirror symmetry for the above examples. The main focus of this article is the higher genus cases, which has not yet been studied in the literature except for genus one in \cite{SatT} and \cite{Str}. Special thanks goes to Milanov-Ruan from which the current work draws a great deal of inspiration. Both of us would like to thank our advisor Prof. Yongbin Ruan for introducing them to the subject and many discussions that made this work possible. We would like to thank Alessandro Chiodo, Huijun Fan, Tyler Jarvis, Todor Milanov for many helpful discussions. We thank Arthur Greenspoon for the editorial assistance.

\section{A review of B-model theory}

  To draw the conclusions to the LG/CY correspondence, we use the results of Milanov-Ruan \cite{MR} on the B-model side. In this section, we review their work briefly. Readers are referred to their paper for the details.

The setting of B-model theory is the Saito-Givental theory for the miniversal deformation of a singularity.
 If $\{\phi_i\}$ forms a homogeneous vector space basis of the Milnor ring $\mathscr{Q}_W=\mathbb{C}[x_1,\cdots,x_N]/\textrm{Jac}(W)$, where $\textrm{Jac}(W)$ is the Jacobi ideal of $W$, then the miniversal deformation space of an isolated hypersurface singularity $W$ is
\begin{equation*}
W(\textbf{s},\textbf{x})=W+\sum_{i=-1}^{\mu-2}s_i \phi_i.
\end{equation*}
$\mu$ is the dimension of the Milnor ring as a vector space.
We assume $W(\textbf{s},\textbf{x})$ is homogeneous and then denote $d_i=\deg\,s_i=1-\deg\phi_i.$

According to \cite{S}, there exists a local Frobenius manifold structure on the base space of a universal unfolding of an isolated hypersurface singularity.
For a generic point in the miniversal deformation space, the Frobenius manifold is semisimple. According to Givental \cite{Gi}, higher genus Gromov-Witten type invariants and ancestor functions can be defined as:
\begin{equation*}
\cal{A}_{W}^{SG}=\exp(\sum_{g\geq0}\hbar^{2g-2}\cal{F}_{g,W}^{SG}).
\end{equation*}
\subsection{Global Saito-Givental theories for elliptic singularities}
The B-models of our elliptic orbifold $\mathbb{P}^1$ are three family of elliptic singularities.
\begin{eqnarray}
\begin{split}
&P_8(\sigma): x^3+y^3+z^3+\sigma xyz.\\
&X_9(\sigma): x^2z+xy^3+z^2+\sigma xyz. \\
&J_{10}(\sigma): x^3z+y^3+z^2+\sigma xyz.
\end{split}
\end{eqnarray}
where the parameter
 space of $\sigma$ is
 $$\Sigma=\{\sigma\in\CC\,|\,\sigma^3+27\neq 0\}.$$
  To solve the LG/CY correspondence, we need to vary the Saito-Givental theory globally from $\sigma=0$ to $\sigma=\infty$.
  The main content  of \cite{MR} is a  construction of such a global Saito-Givental theory.

Here we use the notation as in \cite{MR}, and consider the miniversal deformation of $W_{\sigma}$, where $W_{\sigma}=P_8(\sigma), X_9(\sigma), J_{10}(\sigma)$. For $\textbf{s}=(s_{-1},s_0,\cdots,s_{\mu-2})\in\mathcal{S}= \Sigma\times\mathbb{C}^{\mu-1}$, we define $W_{\sigma}(\textbf{s},\textbf{x}):\mathcal{S}\times\mathbb{C}^3\longrightarrow\mathbb{C}$ by
\begin{equation}
W_{\sigma}(\textbf{s},\textbf{x})=W_{\sigma}+\sum_{i=-1}^{\mu-2}s_i\phi_i
\end{equation}
where $\{\phi_i\}$ forms a homogeneous vector basis of $\mathscr{Q}_{W_{\sigma}}$, the Milnor ring of $W_{\sigma}$.
Then we have the following maps:
\begin{equation*}
\begin{CD}
\mathcal{S}\times \mathbb{C}^3 & & \\
@V{\varphi}VV\searrow &  \\
\mathcal{S}\times \mathbb{C} &@>>{p}> & \mathcal{S}
\end{CD}
\qquad
\begin{tabular}{rl}
&$\varphi(\textbf{s},\textbf{x})= (\textbf{s},W_{\sigma}(\textbf{s},\textbf{x}))$,\\
\\
&$p(\textbf{s},\lambda)=\textbf{s}.$
\end{tabular}
\end{equation*}
For a generic point $(\textbf{s},\lambda)\in\cal{K}\subset\cal{S}\times\mathbb{C}$, the fiber is homotopy equivalent to a bouquet of $\mu$ middle-dimensional spheres.

Consider the bundle $p:\cal{HF}\longrightarrow\cal{K}$ with fiber $X_{\textbf{s},\lambda}$. In \cite{MR}, fixing $(\textbf{s},\lambda)=(s_{-1},0,\cdots,0,\lambda)$, they choose a \emph{toroidal} cycle $a$ which is in the kernel of the intersection pairing on $H^2(X_{s,\lambda},\mathbb{Q})$ such that:
\begin{equation*}
\pi_A(s_{-1})=\pi_a(s):=\int_{a}\frac{d^3x}{dW}
\end{equation*}
According to Saito \cite{S}, the existence of a Frobenius structure on $\cal{S}$ depends on a choice of primitive form
 $\omega=g(\sigma,x)d^3x$, where $g(\sigma, x)$ is homogeneous of degree zero.
 For simple elliptic singularities, it is known that
 $g(\sigma,x)=1/\pi_A(s_{-1})$,
  where
 $\pi_A(s_{-1})$ is the period of a middle-dimensional cycle $A$. Milanov-Ruan's global
 theory depends on a flat family of $A$ and hence $\pi_A(s_{-1})$. It can be constructed
 by choosing $A$ at a base point and extending to a multiple-valued flat family.

 The natural parameter space of global Saito-Givental theory is the universal cover of $\Sigma$, denoted by $\widetilde{\Sigma}$.
 Recall that by fixing a symplectic basis $A',B'$, we have a canonical local isomorphism:
\begin{equation*}
\tau':=\frac{\pi_{B'}(\sigma)}{\pi_{A'}(\sigma)}:\Sigma\longrightarrow\mathbb{H}
\end{equation*}
where $\mathbb{H}$ is the upper half-plane. By analytic continuation of $A', B'$, it induces an isomorphism from $\widetilde{\Sigma}$ to $\mathbb{H}$.
Milanov-Ruan choose $g(\sigma, x)=1/\pi_{A'}(s_{-1})$. They shows that the choice of the $B'$-cycle has theadditional good property that the flat coordinate along $\sigma$-direction also lies in $\mathbb{H}$. Therefore, the parameter space of global Saito-Givental
theory can be chosen naturally as $\tau\in\mathbb{H}$. Later, the choice of $A', B'$ are fixed near $\sigma=\infty$ in order to
match with Gromov-Witten theory. From now on, we will use $\cal{F}_{g,W}(t,\tau), \cal{A}_{W}^{SG}(\tau)$ to denote the dependence on an element
$\tau\in \mathbb{H}$. By definition $\cal{A}_{W}^{SG}(\tau')$ is the analytic continuation of $\tau$.

\begin{rem}
Givental's higher genus function is only defined at semi-simple points $\textbf{s}\neq 0$. However, we need to expand $\cal{A}_{W}^{SG}(\tau)$ at
non-semisimple points. The extension of Givental function to the non-semisimple locus has been a difficult problem in the subject. The keys
to the proof are convergence of FJRW-theory and Gromov-Witten theory and our two mirror theorems (Theorem 1.4 and Theorem 1.5).
\end{rem}

\subsection{Flat coordinates}
The middle cohomology groups $H^2(X_{s,\lambda})$
form a vector bundle equipped with a flat Gauss-Manin connection.
For any relative infinite cycle $\cal{A}\in\,H_3(\mathbb{C}^3, W^{\infty},\mathbb{C})$, we can define the oscillatory integral
\begin{equation*}
J_{\cal{A}}(\textbf{s},z):=(-2\pi z)^{-3/2}zd_{\mathcal{S}}\int_{\cal{A}}e^{W(\textbf{s},\textbf{x})/z}\omega.
\end{equation*}
where $d_{\cal{S}}$ is the de Rham differential on $\cal{S}$.
This gives a family of sections of the cotangent sheaf $T^*_{\cal{S}}$ parameterized by $z$. The oscillatory integral satisfies the following differential equations:
\begin{eqnarray}\label{eq:quantum PDE}
\begin{split}
&z\partial_{i}J_{\cal{A}}(\textbf{t},z)=\partial_{i}\bullet_{\textbf{t}}J_{\cal{A}}(\textbf{t},z),\\
&(z\partial_{z}+E)J_{\cal{A}}(\textbf{t},z)=\theta\ J_{\cal{A}}(\textbf{t},z).
\end{split}
\end{eqnarray}
where $\textbf{t}:=(t_{-1},t_{0},\cdots,t_{\mu-2})$ are flat coordinates, and $\partial_{i}$ is the partial derivative along the $t_i$ direction. $E$ is the \emph{Euler vector field}, and $\theta$ is the \emph{Hodge grading operator} $\theta:T^*_{\cal{S}}\longrightarrow\,T^*_{\cal{S}}$ which sends $dt_i$ to $(\frac{1}{2}-d_i)dt_i$.

The flat coordinates are computed by analyzing the Gauss-Manin connection at $z=\infty$. Rescaling the oscillatory integral, it can be expanded in powers of $z$ as follows:
\begin{equation}
(-2\pi z)^{-3/2}\int_{\mathcal{A}}e^{W(\textbf{s},\textbf{x})/z}\omega
=z^{-1/2}\sum_{\delta}z^{-\delta}\int_{\alpha(\sigma,1)}C_{\delta}(\textbf{s},\textbf{x})\frac{\omega}{df}
\end{equation}
for $(\sigma,1)=(\sigma,0,\cdots,0,1)\in\cal{S}\times\mathbb{C}$, and $\alpha(\sigma,1)$ is a middle-dimensional cycle in $X_{\sigma,1}$. Also here we have $\delta=\sum_{i=0}^{\mu-2}k_{i}d_{i}$ for $k_{i}\geq0,k_{i}\in\mathbb{Z}.$ We have the expansion
\begin{equation}\label{eq:standard-coordinate}
C_{\delta}(\textbf{s},\textbf{x})=\sum\widetilde{\Gamma}\Big(\sum_{i=0}^{\mu-2}k_{i}q_{i}\Big)\prod_{i=0}^{\mu-2}\frac{\Big(s_{i}\phi_{i}\Big)^{k_i}}{k_{i}!}.
\end{equation}
where $\widetilde{\Gamma}(k)=(2\pi)^{-3/2}e^{\pi\,i(k-1/2)}\int_{0}^{\infty}e^{-u}u^kdu.$
On the other hand, given a middle-dimensional homology cycle $\alpha_{i}(\sigma,1)\in\,H_2(X_{\sigma,1})$, we can consider the period matrix
\begin{equation}\label{eq:period matrix}
\Pi_{\delta,i}=\int_{\alpha_{i}(\sigma,1)}C_{\delta}(\textbf{s},\textbf{x})\frac{\omega}{df},i=-1,0,\cdots,\mu-2
\end{equation}
According to Milanov-Ruan, by analyzing the quantum differential equations (\ref{eq:quantum PDE}) along $z$ near $\infty$, there exist cycles $\alpha_i(\sigma,1)$ such that the entries $\Pi_{\delta,i}$ are as follows:
\begin{equation}\label{eq:flat coord}
\Pi_{\delta,i}
=\left\{
\begin{array}{ll}
1, &\,\text{if}\,\delta=i=0\\
\frac{1}{2}\sum_{j=-1}^{\mu-2}t_jt_j', &\, \text{if}\,\delta=1,i=-1,\\
t_i,&\,\text{if}\,\delta=d_i,\\
0,& \text{otherwise}.
\end{array}
\right.\end{equation}
In particular, $\Pi_{0,0}=1$ and it implies $$\pi_{A}(s_{-1})=\int_{\alpha_{0}(\sigma,1)}\widetilde{\Gamma}(0)\frac{d^3x}{df}.$$
The degree zero flat coordinate has the expression:
\begin{equation*}
t_{-1}=\frac{\pi_{B}(s_{-1})}{\pi_{A}(s_{-1})}.
\end{equation*}
The oscillatory integrals satisfy the Gauss-Manin connection properties:
\begin{eqnarray}\label{eq:Guass-Manin}
\begin{split}
&\nabla_{\partial/\partial s_i}\int\phi(\textbf{s},\textbf{x})W_{x_i}\frac{d^3x}{dW}
=-\int\phi(\textbf{s},\textbf{x})\frac{\partial\,W}{\partial\,s_i}\frac{d^3x}{dW}\\
&\nabla_{\partial/\partial\lambda}\int\phi(\textbf{s},\textbf{x})W_{x_i}\frac{d^3x}{dW}
=\int\partial_{x_i}\phi(\textbf{s},\textbf{x})\frac{d^3x}{dW}
\end{split}
\end{eqnarray}
Using the Gauss-Manin connection, we can expand the period matrix as a power series in the canonical coordinates $s_i$'s. Comparing with the period matrix (\ref{eq:period matrix}) will give the transition formula between the canonical coordinates and the flat coordinates.

\subsection{Gepner limit versus large complex structure limit}
The Gromov-Witten theory is expected to match with the expansion of the Saito-Givental function at $\sigma=\infty$ with a choice of
symplectic basis $A_{\infty}, B_{\infty}$ such that local monodromy is maximally unipotent. This is referred as a {\em large complex structure limit}.
We shall fix this particular choice of $A_{\infty}, B_{\infty}$
and analytic continue it to a flat family. After identifying with the $\tau$ coordinate, $\sigma=\infty$ correspond to $\tau=i\infty$.
We denote the Saito-Givental theory at the large complex structure limit by $\cal{A}^{SG}_{W_{\infty}}$.
The matching between $\cal{A}^{SG}_{W_{\infty}}$ and Gromov-Witten theory is called LG-to-CY mirror symmetry.

We can analytic continue $\cal{A}^{SG}_{W_{\infty}}$ from $\tau$ near $i\infty$ to the region near $\sigma=0$ (still denoted by $\cal{A}^{SG}_{W_{\infty}}$). Again, we obtain a
Saito-Givental theory for a symplectic integral basis denoted also by $A_{\infty}, B_{\infty}$ near $\sigma=0$. However, FJRW-theory does NOT corresponds to any integrable basis. To match with
FJRW-theory (LG-to-LG mirror symmetry),
We can choose a complex basis
\begin{eqnarray}
\begin{split}
&\pi_{A_0}(s_{-1})=\ _2F_1(\frac{1}{3},\frac{1}{3};\frac{2}{3};-\frac{s_{-1}^3}{27});\\
&\pi_{B_0}(s_{-1})=\ _2F_1(\frac{2}{3},\frac{2}{3};\frac{4}{3};-\frac{s_{-1}^3}{27})s_{-1}.\\
\end{split}
\end{eqnarray}
\begin{rem}
For the Fermat type singularity $P_8$, this choice coincides with the one computed in \cite{NY}.
\end{rem}
This basis has the property of diagonalizing a certain monodromy. It was referred as the {\em Gepner limit}.
We denote the Saito-Givental theory of the Gepner limit by $\cal{A}^{SG}_{W}$. Let ${\mathbb U}$ be the
change of basis matrix between $A_0, B_0$ and $A_{\infty}, B_{\infty}$. There is a recipe to construct
a differential operator $\widehat{{\mathbb U}}$ from ${\mathbb U}$ called the {\em quantization of ${\mathbb U}$}. The central theorem (Theorem 4.4 in \cite{MR}) we used is
\begin{thm}
Up to an overall constant,
$$\widehat{{\mathbb U}}(\cal{A}^{SG}_{W})=\cal{A}^{SG}_{W_{\infty}}.$$
\end{thm}

\section{LG-to-LG mirror theorem of all genera}

In this section, we establish the LG-side of global mirror symmetry. They are (1) convergence of FJRW-theory; (2) LG-to-LG mirror
theorem of all genera.

\subsection{Fan-Jarvis-Ruan-Witten theory}

We first review FJRW-theory with a focus on elliptic singularities.  For more details about FJRW-theory, see \cite{FJR2}.

\subsubsection{The singularity $(W,G)$ and its FJRW state space}
\begin{df}
$W$ is a \emph{quasi-homogeneous non-degenerate polynomial} if $W$ satisfies the following properties:
\begin{enumerate}
\item (Quasihomogeneity) There exist weights $q_i\in\mathbb{Q}$, such that for all $\lambda\in\mathbb{C}^*$,
\begin{equation*}
W(\lambda^{q_1}x_1,\cdots,\lambda^{q_N}x_N)=\lambda W(x_1,\cdots,x_N).
\end{equation*}
\item The choice of the weights $q_i$ is unique.
\item $W$ has an isolated singularity only at 0.
\end{enumerate}
\end{df}
For each non-degenerate $W$, we have a symmetry group which we call the \emph{maximal diagonal symmetry group} of $W$, and denote by $G_{\textrm{max}}$:
\begin{equation*}
G_{\textrm{max}}:=\Big\{\gamma=(\lambda_1,\cdots,\lambda_N)\in(\mathbb{C}^*)^N\Big|W(\lambda_1x_1,\cdots,\lambda_N x_N)=W(x_1,\cdots,x_N)\Big\}.
\end{equation*}
$G_{\textrm{max}}$ always contains the \emph{exponential grading element} $J:=(e^{2\pi iq_1},\cdots,e^{2\pi iq_N})$.
In addition to $W$, FJRW theory depends on a choice of a suitable group $G$ with $\LD J\RD \subset G \subset G_{\textrm{max}}$.

The \emph{central charge} of $W$ is defined to be
\begin{equation*}
\hat{c}_{W}:=\sum_{i=1}^{N}(1-2q_i).
\end{equation*}
For $\hat{c}_{W}<1$, $W$ is called a \emph{simple singularity} and has been completely classified as the famous ADE-singularities. For $\hat{c}_{W}=1$, $W$ is called an \emph{elliptic singularity}.
In this paper, we study the FJRW theory of three special elliptic singularities with symmetry group $G=G_{\textrm{max}}$.
In the following, our $G$ equals $G_{\textrm{max}}$ unless otherwise stated.

For any $\gamma\in\,G$, we denote by $\mathbb{C}_{\gamma}^{N_{\gamma}}$ the fixed points of $\gamma$, where $N_{\gamma}$ is the complex dimension of the fixed locus. We denote by $W_{\gamma}$ the restriction of $W$ to the fixed locus. According to \cite{FJR2}, $W_{\gamma}$ is also non-degenerate.
\begin{df}
We define the \emph{$\gamma$-twisted sector} $H_{\gamma}$ to be the $G$-invariant part of the relative middle-dimensional cohomology for $W_{\gamma}$,
\begin{equation*}
H_{\gamma}:=H^*(\mathbb{C}^{N_{\gamma}}_{\gamma};W^{\infty}_{\gamma};\mathbb{C})^G
\end{equation*}
Here $W^{\infty}_{\gamma}=(Re\,W_{\gamma})^{-1}(M,\infty), M\gg0.$

The \emph{FJRW state space} $H^*_{FJRW}(W,G)$ is defined to be the direct sum of all \emph{$\gamma$-twisted sectors $H_{\gamma}$} for the pair $(W,G)$:
\begin{equation}
H^*_{FJRW}(W,G):=\bigoplus_{\gamma\in G}H_{\gamma}=\bigoplus_{\gamma\in G} H^*(\mathbb{C}^{N_{\gamma}}_{\gamma};W^{\infty}_{\gamma};\mathbb{C})^G
\end{equation}
\end{df}
For any $\gamma\in G$, we have $\gamma=(\exp(2\pi i\Theta_1^{\gamma}),\cdots,\exp(2\pi i\Theta_N^{\gamma}))\in(\mathbb{C}^*)^N$ for some unique $\Theta_i^{\gamma}\in[0,1)\cap\mathbb{Q}$.
We define the \emph{degree shifting number} $\iota_{\gamma}$ by:
\begin{equation*}
\iota_{\gamma}:=\sum_{i=1}^N(\Theta_i^{\gamma}-q_i).
\end{equation*}
For each homogeneous element $\alpha\in H_{\gamma},$ the \emph{complex degree} $\deg_W\alpha$ is defined as:
\begin{equation*}
\deg_W\alpha:=\deg\alpha+\iota_{\gamma}
\end{equation*}
where $\deg\alpha$ is the ordinary complex degree as an relative cohomology class. As an consequence, $H^*_{FJRW}(W,G)$ is a graded vector space under this grading.
\begin{df}
We say the $\gamma$-twisted sector $\alpha$ is \emph{narrow} if $N_{\gamma}=0$. Otherwise, we say the $\gamma$-sector is \emph{broad}.
\end{df}
For each $\gamma\in\,G$, there is a natural intersection pairing $\LD,\RD_{\gamma}$:
\begin{equation*}
\LD,\RD_{\gamma}:H_{\gamma}\times\,H_{\gamma^{-1}}\longrightarrow\mathbb{C}
\end{equation*}
Note that $H_{\gamma^{-1}}=H_{\gamma}$. For any $W$, the pairing on a narrow sector is obvious since it is one-dimensional. For a broad sector $\gamma$,
there is an isomorphism between $\mathscr{Q}_{W_{\gamma}}$ and $H^{N_{\gamma}}(\mathbb{C}^{N_{\gamma}}_{\gamma},W^{\infty}_{\gamma},\mathbb{C})$. The pairing for the broad sector is isomorphic to the residue pairing on $\mathscr{Q}_{W_{\gamma}}$.
The residue pairing is:
\begin{equation}\label{eq:Res}
\LD f,g\RD:=\textrm{Res}_{\textbf{x}=0}\frac{fg dx_1\cdots dx_N}{\frac{\partial W}{\partial x_{1}}\cdots
\frac{\partial W}{\partial {x_N}}}=C\mu.
\end{equation}
$C$ is the unique constant such that $fg=C\cdot\textrm{Hessian}(W)\mod \mathrm{Jac}(W)$.
The pairing of the FJRW space is defined to be
\begin{equation*}
\LD,\RD:=\sum_{\gamma\in\,G}\LD,\RD_{\gamma}
\end{equation*}

\subsubsection{W-Structure, virtual cycle and Cohomological Field Theory}
We start with a genus $g$ possibly nodal orbi-curve $C$ with $k$ marked points $p_1,\cdots,p_k$, and its \emph{log-canonical bundle}, defined by
\begin{equation*}
\omega_{C,\log}:=\omega_{C}\otimes\mathscr{O}(p_1)\otimes\cdots\otimes\mathscr{O}(p_k).
\end{equation*}
where $\omega_{C}$ is the canonical bundle of $C$, $\mathscr{O}(p_i)$ is the holomorphic line bundle of degree one whose sections may have
a simple pole at $p$.
We write the polynomial $W=\sum_{j=1}^{s}W_j$ as a sum of monomials for $W_{j}=c_j\prod_{i=1}^{N}x_{i}^{b_{j,i}}$. A $W$-structure on the curve $C$ is a choice of $N$ orbifold line bundles $\cal{L}_1,\cdots,\cal{L}_N$ and $s$ isomorphisms of line bundles
\begin{equation*}
\varphi_j:W_j(\cal{L}_1,\cdots,\cal{L}_N)=\cal{L}_1^{\otimes\,b_{j,1}}\otimes\cdots\otimes\cal{L}_N^{\otimes\,b_{j,N}}\longrightarrow\,K_{\log}
\end{equation*}
We denote the moduli space of $W$-structures by $\mathscr{W}_{g,k}$. By \cite{FJR2}, the orbifold structure at a marked point
(or node) is specified by a group element $\gamma\in G_{\textrm{max}}$. By fixing the orbifold decorations at marked point, we has the decomposition:
\begin{equation*}
\mathscr{W}_{g,k}=\sum_{\boldsymbol\gamma}\mathscr{W}_{g,k}(\boldsymbol\gamma)
\end{equation*}
Here $\boldsymbol\gamma=(\gamma_1,\cdots,\gamma_k)$ is the orbifold decorations at marked points.
If we forget both the $W$-structure and the orbifold structure, then we have the forgetful morphism:
\begin{equation*}
\textrm{st}:\mathscr{W}_{g,k}\longrightarrow\overline{\mathscr{M}}_{g,k}
\end{equation*}
$\overline{\mathscr{M}}_{g,k}$ is the Deligne-Mumford stack of stable curves..

For each $W$-structure $(C,p_1,\cdots,p_k,\cal{L}_1,\cdots,\cal{L}_N,\varphi_1,\cdots,\varphi_s)$, we can assign a \emph{$G$-decorated dual graph} $\Gamma$. $\Gamma$ is the dual graph of the underlying curve, with each vertex representing an irreducible component of the underlying curve, and each edge representing a node, each half-edge representing a marked point. For each vertex $\gamma$, let $g_{\gamma}$ be the genus of the corresponding component. And for each half-edge $\tau$ of $\Gamma$, we assign an element $\gamma_{\tau}\in\,G$. If we assign each edge a pair of decorations $\gamma_+\in\,G$ and $\gamma_{+}^{-1}$, then we say that the dual graph $\Gamma$ is \emph{fully $G$-decorated}.

For each $\Gamma$, the moduli of $W$-structures with $\Gamma$ the corresponding $G$-decorated dual graph is also a stack, which we denote by $\mathscr{W}(\Gamma)$. We have the canonical inclusion map: $i:\mathscr{W}(\Gamma)\longrightarrow\mathscr{W}_{g,k}(\boldsymbol\gamma)$.
In \cite{FJR2}, the authors constructed a \emph{virtual fundamental cycle}:
\begin{equation*}
[\mathscr{W}(\Gamma)]^{\mathrm{vir}}\in\,H_{*}(\mathscr{W}(\Gamma),\mathbb{Q})\otimes\prod_{\tau\in\,T(\Gamma)}H_{N_{\gamma_{\tau}}}(\mathbb{C}_{\gamma_{\tau}}^{N_{\gamma_{\tau}}},W_{\gamma_{\tau}}^{\infty},\mathbb{Q})^G
\end{equation*}
where $\gamma_{\tau}\in\,G$.
$[\mathscr{W}_{g,k}(\boldsymbol\gamma)]^{\mathrm{vir}}$ is defined in the similar fashion. And we have
\begin{equation*}
[\mathscr{W}(\Gamma)]^{\mathrm{vir}}=i^*\,[\mathscr{W}_{g,k}(\boldsymbol\gamma)]^{\mathrm{vir}}
\end{equation*}
Using the virtual cycle, they defined \emph{a cohomological field theory}
$$\Lambda_{g,k}^{W,G}:(H^*_{FJRW})^{\otimes k}\longrightarrow H^*(\overline{\mathscr{M}}_{g,k})$$
by
\begin{equation*}
\Lambda_{g,k}^{W,G}(\boldsymbol\alpha):=\frac{|G|^g}{\deg(\textrm{st})}\textrm{PD}\,\textrm{st}_*\Big([\mathscr{W}_{g,k}(\gamma)]^{\mathrm{vir}}\cap\prod_{i=1}^{k}\alpha_i\Big)
\end{equation*}
for $\alpha_i\in\,H_{\gamma_i}$, and then extended linearly to the whole space. In this paper, we always fix a basis of $H^*_{FJRW}$ and choose those $\alpha_i$'s in the fixed basis. Here, $\textrm{PD}$ is the Poincar\'{e} duality map. The \emph{FJRW ancestor correlators} for $(W,G)$ are defined to be
\begin{equation}
\LD\tau_{l_1}(\alpha_1),\cdots,\tau_{l_k}(\alpha_k)\RD_g^{W,G}
=\int_{\overline{\mathscr{M}}_{g,k}}\Lambda_{g,k}^{W,G}(\alpha_1,\cdots,\alpha_k)\prod_{i=1}^k\psi_i^{l_i}
\end{equation}
The correlator is called \emph{primary} if all $l_i=0$.
We assign a formal variable $t_{i, l}$ for the insertion
$\tau_l(\alpha_i)$ and the \emph{FJRW generating function} of $(W,G)$ is defined as follows:
\begin{equation}
\cal{F}_{g,W,G}=\sum_{n}\sum_{l_i}\sum_{j_i}\LD\tau_{l_1}(\alpha_{j_1}),\cdots,\tau_{l_n}(\alpha_{j_k})\RD_g^{W,G}
\frac{\prod_{i}t_{j_i,l_i}}{\prod_{i} l_i!}
\end{equation}
The \emph{total ancestor potential} of FRJW theory for $(W,G)$ is defined to be
\begin{equation}
\cal{A}_{FJRW}^{W,G}=\exp(\sum_{g\geq0}\hbar^{2g-2}\cal{F}_{g,W,G})
\end{equation}
In this paper, we also consider the {\em FJRW ancestor correlator function} at $\textbf{t}$,
\begin{equation}
\LD\LD\tau_{l_1}(\alpha_1),\cdots,\tau_{l_k}(\alpha_k)\RD\RD_g^{W,G}(\textbf{t})
=\sum_n\frac{1}{n!}\int_{\overline{\mathscr{M}}_{g,n+k}}\Lambda_{g,k+n}^{W,G}(\alpha_1,\cdots,\alpha_k, \textbf{t}, \cdots,\textbf{t})\prod_{i=1}^k\bar{\psi}_i^{l_i}
\end{equation}
where $\bar{\psi}_i=\pi^*_{k,n}(\psi_i)$ is pulling back the $\psi$-classes by the forgetful morphism
$\pi_{k,n}: \overline{\mathscr{M}}_{g,k+n}\rightarrow \overline{\mathscr{M}}_{g,k}$. We denote its generating function and corresponding ancestor potential function by $\cal{F}_{g,W,G}(\textbf{t}), \cal{A}_{FJRW}^{W,G}(\textbf{t})$.

\subsubsection{Some properties of $[\mathscr{W}_{g,k}(\gamma)]^{\mathrm{vir}}$.}
Now let us fix $\Gamma$ to be a $G$-decorated dual graph of genus $g$, $k$ marked point orbi-curve with only one irreducible component, and with half-edges decorated by $\gamma_1,\cdots,\gamma_k$. We list some properties of  $[\mathscr{W}_{g,k}(\gamma)]^{\mathrm{vir}}$
which are useful to us. We refer readers to \cite{FJR2} for a full set of axioms.
\begin{enumerate}
\item (Selection rule)
The complex degree of the virtual fundamental cycle $[\mathscr{W}(\Gamma)]^{\mathrm{vir}}$ is $3g-3+k-D$, where $D$ is the codimension:
\begin{equation*}
D:=-\sum_{i=1}^{N}\textrm{index}(\cal{L}_i)=\hat{c}_{W}(g-1)+\sum_{\tau\in\,T(\Gamma)}\iota_{\gamma_{\tau}}
\end{equation*}
Thus for elliptic singularities, $\LD\tau_{l_1}(\alpha_1),\cdots,\tau_{l_k}(\alpha_k)\RD_g^{W,G}$ is nonzero unless
\begin{equation}\label{deg:FJRW}
\sum_{i=1}^k\deg_W(\alpha_i)+\sum_{i=1}^k l_i=2g-2+k
\end{equation}

\item (Line bundle criterion). If the moduli space $\mathscr{W}_{g,n}(\boldsymbol\gamma)$ is non-empty, then the degree of the desingularized line bundle $|\cal{L}_j|$ is an integer, i.e.
\begin{equation}\label{deg:Line BD}
\deg(|\cal{L}_j|)=\Big(q_j(2g-2+k)-\sum_{l=1}^{k}\Theta_{i}^{\gamma_l}\Big)\in\mathbb{Z}
\end{equation}

\item (Concavity) Suppose all marked points are narrow, and the concavity condition $\pi_{*}\bigoplus_{i=1}^{N}\cal{L}_i=0$ holds; then
\begin{equation}\label{eq:Concave}
[\mathscr{W}_{g,k}(\boldsymbol\gamma)]^{\mathrm{vir}}
=c_{\textrm{top}}\Big(-R^1\pi_*\bigoplus_{i=1}^N\cal{L}_i\Big)\cap[\mathscr{W}_{g,k}(\boldsymbol\gamma)]
\end{equation}
\end{enumerate}

\subsubsection{Mirror symmetry of Frobenius algebra}
The three point correlators and pairing define a Frobenius algebra structure, where the multiplication is given by $$\LD\alpha\star\beta,\gamma\RD=\LD\alpha,\beta,\gamma\RD_{0,3}^{W,G}.$$
It follows from \cite{Kr} that the FJRW-Frobenius
algebras of $P_8, X^T_9, J^T_{10}$ are isomorphic to Milnor rings of $P_8, X_9, J_{10}$. In this section, we review the results for the purpose of setting up the notations such as the canonical basis. Readers are refered to \cite{Kr} for the details of computation.

{\bf $P_8$-Case: }

Now we consider the case of $P_8:=x^3+y^3+z^3$. The weights are $q_x=q_y=q_z=1/3.$
Then $G_{\textrm{max}}=\mathbb{Z}_3\times\mathbb{Z}_3\times\mathbb{Z}_3$.
We denote $$\gamma_x=\textrm{Diag}(\omega^2,\omega,\omega),\gamma_y=\textrm{Diag}(\omega,\omega^2,\omega),\gamma_z=\textrm{Diag}(\omega,\omega,\omega^2),$$
where $\omega=e^{2\pi i/3}$. These are also generators in $G_{\textrm{max}}$. We denote
\begin{equation*}
\be_{x}:=1\in H^{\textrm{mid}}(\mathbb{C}_{\gamma_x}^{N^{\gamma_x}},W_{\gamma_x}^{\infty},\mathbb{Q})
\end{equation*}
$H_{FJRW}^{P_8,G_{\textrm{max}}}$ is spanned by a canonical basis
\begin{equation*}
\textbf{1}=\be_{J},\be_{x},\be_{y},\be_{z},\be_{xy},\be_{xz},\be_{yz},\be_{J^{-1}}=\be_{xyz}.
\end{equation*}
where the subscripts correspond to the square of $\omega$ which shows up in the diagonal element of the group $G_{\textrm{max}}$. We have the table:
\begin{center}
\begin{tabular}{l||l|l|l|l|l|l|l|l } $\gamma=$&$\bf{1}$&$\be_{x}$&$\be_{y}$&$\be_{z}$&$\be_{xy}$&$\be_{xz}$&$\be_{yz}$&$\be_{J^{-1}}$\\\hline
  $\Theta_x^{\gamma}$&1/3         & 2/3   &1/3    &1/3    &2/3     &2/3     &1/3     &2/3     \\\hline
  $\Theta_y^{\gamma}$&1/3         & 1/3   &2/3    &1/3    &2/3     &1/3     &2/3     &2/3     \\\hline
  $\Theta_z^{\gamma}$&1/3         & 1/3   &1/3    &2/3    &1/3     &2/3     &2/3     &2/3     \\\hline
  $\deg_{W}$&0           &1/3    &1/3    &1/3    &2/3     &2/3     &2/3     &1       \\
\end{tabular}
\end{center}
Up to symmetry, all the nonzero genus-0 3-point correlators are
\begin{equation*}
\LD\textbf{1},\textbf{1},\be_{xyz}\RD_{0}^{P_8}=\LD\textbf{1},\be_{x},\be_{yz}\RD_{0}^{P_8}=\LD\be_{x},\be_{y},\be_{z}\RD_{0}^{P_8}=1.
\end{equation*}
The map $\Psi_{P_8}:\mathbb{C}[X,Y,Z]\longrightarrow H_{FJRW}^{P_8,G_{\textrm{max}}}$ is surjective with kernel containing $X^2, Y^2$ and $Z^2$ and defines a graded algebra isomorphism $\Psi_{P_8}:\mathscr{Q}_{P_8}\longrightarrow H_{FJRW}^{P_8,G_{\textrm{max}}}$.

{\bf $X_{9}^T$-Case}:

Here $X_{9}^T:=x^2y+y^3+xz^2$.
The weights are $q_x=q_y=q_z=1/3.$ $G_{\textrm{max}}$ is generated by
$\gamma=\textrm{Diag}(\xi^{10},\xi^{4},\xi)$,
where $\xi=e^{\pi i/6}$, is a 12-th root of unity.
As a vector space, $H_{FJRW}^{X_{9}^T,G_{W}}$ is generated by a canonical basis
\begin{equation*}
x\be_0, \be_1,\be_2,\be_4,\be_5, x\be_0, \be_7, \be_8, \be_{10}, \be_{11}.
\end{equation*}
where $\be_{i}:=1\in H^{\textrm{mid}}(\mathbb{C}_{\gamma^i}^{N^{\gamma^i}},W_{\gamma^i}^{\infty},\mathbb{Q})
$. $\be_0:=dx\wedge dy$, and $x\be_0$ is a broad sector. We list the numerical invariants of the generators as follows:
\begin{center}
\begin{tabular}{l||l|l|l|l|l|l|l|l|l }    $\gamma=$&$\be_{1}$&$\be_{2}$&$\be_{4}$&$\be_{5}$&$x\be_0$&$\be_{7}$&$\be_{8}$&$\be_{10}$&$\be_{11}$\\\hline
  $\Theta_x^{\gamma}$&5/6&2/3      &1/3        &1/6      &0       &5/6         &2/3        &1/3         &1/6 \\\hline
  $\Theta_y^{\gamma}$&1/3  &2/3      &1/3        &2/3    &0        &1/3        &2/3        &1/3         &2/3  \\\hline
  $\Theta_z^{\gamma}$&1/12      &1/6        &1/3      &5/12       &1/2           &7/12        &2/3         &5/6  &11/12 \\\hline
  $\deg_{W}$&1/4&1/2      &0          &1/4      &1/2       &3/4         &1        &1/2         &3/4 \\
\end{tabular}
\end{center}
Here the identity is $\be_{4}=\textbf{1}$. The nonzero pairings are $\LD\be_{i},\be_{12-i}\RD=1, \LD\,x\be_0,x\be_0\RD=-\frac{1}{2}.$
All nonzero genus-0 3-point correlators are:
\begin{itemize}
\item $\LD\mathbf{1},\be_{i},\be_{12-i}\RD_{0}^{X_9^T}=1, \LD\mathbf{1},x\be_0,x\be_0\RD_{0}^{X_9^T}=-\frac{1}{2}$;\\
\item $\LD\be_{1},\be_{1},\be_{2}\RD_{0}^{X_9^T}=-2,\LD\be_{1},\be_{5},\be_{10}\RD_{0}^{X_9^T}=1,\LD\be_{5},\be_{5},x\be_0\RD_{0}^{X_9^T}=\pm 1.$
\end{itemize}
The map $\Psi_{X_9}:\mathbb{C}[X,Y,Z]\longrightarrow H_{FJRW}^{X_{9}^T,G_{\textrm{max}}}$ given by
\begin{equation*}
X\longmapsto \be_{1}, Y\longmapsto \be_{5}, Z\longmapsto \be_{10}.
\end{equation*}
defines a graded algebra isomorphism from the Milnor ring $\mathscr{Q}_{X_9}$ to $H_{FJRW}^{X_{9}^T,G_{\textrm{max}}}$.

{\bf $J_{10}^T$-Case:}

Here $J_{10}^T=x^3+y^3+xz^2.$ Again, $q_x=q_y=q_z=1/3.$
The group elements are $\gamma_{i,j}=\textrm{Diag}(\omega^{2j},\omega^{i},\xi^j)$, where $\omega=e^{2\pi i/3}, \xi=e^{\pi i/3}$. Hence
$\gamma_{i,j}\in\mathbb{Z}_3\times\mathbb{Z}_6=G_{\textrm{max}}.$
The state space is
  \begin{equation*}
  H_{FJRW}^{J_{10}^T,G_{\textrm{max}}}=\bigoplus_{i=1,2;0\leq j\leq5, j\neq3.}\mathbb{C}\{\gamma_{i,j}\}
  \end{equation*}
For this case, the ring structure of $H_{FJRW}^{J_{10}^T,G_{\textrm{max}}}$ can also be computed using the Sums of Singularities Axiom in \cite{FJR2}. We have
\begin{equation*}
H_{FJRW}^{J_{10}^T,G_{\textrm{max}}}\cong H_{FJRW}^{A_2,G_{A_2}}\otimes H_{FJRW}^{D_4,G_{D_4}}
\end{equation*}
Thus $H_{FJRW}^{J_{10}^T,G_{\textrm{max}}}$ is spanned by a canonical basis
\begin{equation*}
z\be_{0},\be_1,\be_2,\be_4,\be_5,z\be_{6},\be_7,\be_8,\be_{10},\be_{11}.
\end{equation*}
where $\be_{i,j}=\be_{6(i-1)+j}$, $\be_{i,j}$ is the generator of $H^{\textrm{mid}}(\mathbb{C}_{\gamma_{i,j}}^{N^{\gamma_{i,j}}},W_{\gamma_{i,j}}^{\infty},\mathbb{Q})$. Here $z\be_{0}$ and $z\be_{6}$ are both broad sectors, others are narrow sectors.
The numerical invariants are
\begin{center}
\begin{tabular}{l||l|l|l|l|l|l|l|l|l|l }
     $\gamma=$&$z\be_{0}$&$\be_{1}$&$\be_{2}$&$\be_{4}$&$\be_{5}$&$z\be_{6}$&$\be_7$&$\be_8$&$\be_{10}$&$\be_{11}$\\\hline
  $\Theta_x^{\gamma}$&0  &2/3      &1/3        &2/3      &1/3       &0           &2/3        &1/3         &2/3  &1/3 \\\hline
 $\Theta_y^{\gamma}$&1/3&1/3      &1/3        &1/3      &1/3       &2/3         &2/3        &2/3         &2/3  &2/3 \\\hline
  $\Theta_z^{\gamma}$&0  &1/6      &1/3        &2/3      &5/6       &0           &1/6        &1/3         &2/3  &5/6 \\\hline
  $\deg_{W}$&1/3&1/6      &0          &2/3      &1/2       &2/3         &1/2        &1/3         &1   &5/6 \\
\end{tabular}
\end{center}
The ring identity is $\be_{2}=\mathbf{1}$. All the nonzero genus-0 3-point correlators are
\begin{itemize}
\item $\LD\mathbf{1},z\be_{0},z\be_{6}\RD_{0}^{J_{10}^T}=-\frac{1}{2}, \LD\mathbf{1},\be_{i},\be_{12-i}\RD_{0}^{J_{10}^T}=1;$
\item $\LD z\be_{6},\be_{1},\be_{1}\RD_{0}^{J_{10}^T}=
\LD z\be_{0},\be_{1},\be_7\RD_{0}^{J_{10}^T}=\pm 1,
\LD z\be_{0},z\be_{0},\be_8\RD_{0}^{J_{10}^T}=-\frac{1}{2},
\LD \be_{1},\be_{5},\be_8\RD_{0}^{J_{10}^T}=1.$
\end{itemize}
The map $\Psi_{J_{10}}:\mathbb{C}[X,Y,Z]\longrightarrow H_{FJRW}^{J_{10}^T,G_{\textrm{max}}}$ given by
\begin{equation*}
X\longmapsto \be_{1}, Y\longmapsto \be_8, Z\longmapsto \be_{5}.
\end{equation*}
defines a graded algebra isomorphism
$\Psi_{J_{10}}:\mathscr{Q}_{J_{10}}\longrightarrow H_{FJRW}^{J_{10}^T,G_{\textrm{max}}}.$

\subsection{Reconstruction}

   In this section, we introduce the main technique of the paper, namely reconstruction by tautological relations. It has been known for a long time that Gromov-Witten theory poses many tautological relations such as the WDVV relation. Our main observation is that they are sufficiently strong in our case to determine the entire theory from 3-point and some {\em basic} 4-point correlators.
The starting point of the reconstruction is the 3-point correlators of the FJRW-theory as well as the Saito-Givental theory of the mirror. The 3-point correlators of Saito-Givental theory at $W=W_0$ are structural constants of the Milnor ring. They have been identified with 3-point correlators of FJRW-theory by the first author. Then, our reconstruction works simultaneously for both the FJRW-theory (A-model) and the Saito-Givental theory (B-model) of the mirror. To simplify the notation, we work out the details for the FJRW-theory.

Recall that Givental theory is only defined over semisimple points. Hence, we have to work over the ancestor correlators. For all three cases, FJRW-state space contains a unique complex degree one element $\be_{J^{-1}}$. Let $t_{i\geq0}$ be a class of complex degree $<1$. We define the ancestor correlator
$$\langle\langle\tau_{l_1}(\alpha_1),\cdots,\tau_{l_k}(\alpha_k)\rangle\rangle_{g,k}(t_{i\geq0})=\sum_n\frac{1}{n!}\langle \tau_{l_1}(\alpha_1), \cdots, \tau_{l_k}(\alpha_k), t_{i\geq0},\cdots,t_{i\geq0}\rangle_{g,k+n}.$$
By the degree constraint, if we expand $t_{i\geq0}$ as linear combination of basis, the above correlator is a polynomial of coordinate and hence well-defined. Furthermore, the big quantum cohomology is semisimple for a generic value of $\textbf{t}$. Our reconstruction theorem applies to this general ancestor correlator.  On semi-simple point, the reconstruction also follows by \cite{T}. It is important to mention that FJRW-theory is well-defined at $\textbf{t}=0$ which is not semi simple.
In fact, we use this fact to prove that Saito-Givental theory can be extended to $\textbf{t}=0$. Our reconstruction theorem  applies FJRW-theory at $\textbf{t}=0$.
The later is crucial for our convergence theorem.

\subsubsection{Genus-0 reconstruction}

In genus-0, both FJRW-theory and Saito theory are well-defined for $\textbf{t}=0$. The ancestor correlators can obviously be expressed by ordinary correlators with $\textbf{t}=0$.

\begin{thm}\label{thm:FJRW g=0 reconstruction}
Using WDVV equations, all genus-0 primary correlators of FJRW theory for the elliptic singularities $P_8,X_9^T,J_{10}^T$ are uniquely determined by the pairing, the 3-point correlators and the following 4-point correlators:
$\LD\be_{x},\be_{x},\be_{x},\be_{xyz}\RD_{0}^{P_8}$, $\LD\be_{x},\be_{y},\be_{z},\be_{xyz}\RD_{0}^{P_8}$,
   $\LD\be_{1},\be_{5},\be_{7},\be_{7}\RD_{0}^{X_{9}^T}$,
    $\LD\be_{1},\be_{8},\be_{5},\be_{10}\RD_{0}^{J_{10}^T}, \LD\be_{8},\be_{8},\be_{8},\be_{10}\RD_{0}^{J_{10}^T}$ and $ \LD\be_{1},\be_{1},\be_{4},\be_{10}\RD_{0}^{J_{10}^T}.$
\end{thm}
We first introduce some useful concepts.
\begin{df}
We call a homogeneous element $\gamma$ is \emph{primitive} if it cannot be written as $\gamma=\gamma_1\star\gamma_2$ for $\deg_{W}\gamma_1$ and $\deg_{W}\gamma_2$ nonzero.
\end{df}
\begin{df}
We call a genus-0 correlator a \emph{basic correlator} if there are at most two non-primitive insertions, neither of which are 1.
\end{df}
Our general scheme is following recursion formula from the WDVV equation.
\begin{lm}\label{lm:FJRW-primitive}
We can reconstruct genus-0 primary correlators of FJRW theory for the three cases by basic correlators with at most four marked points.
\end{lm}
\begin{proof}
We recall the WDVV equation for the FJRW invariant,
\begin{eqnarray}\label{eq:WDVV-FJRW}
\begin{split}
\LD\gamma_1,\gamma_2,\delta_S,\gamma_3\star\gamma_4\RD_{0,n+3}=&I-\LD\gamma_1\star\gamma_2,\delta_S,\gamma_3,\gamma_4\RD_{0,n+3}\\
&+\LD\gamma_1\star\gamma_3,\delta_S,\gamma_2,\gamma_4\RD_{0,n+3}
+\LD\gamma_1,\gamma_3,\delta_S,\gamma_2\star\gamma_4\RD_{0,n+3}
\end{split}
\end{eqnarray}
where $S=S(n):=\{1,\cdots,n\}, \delta_A:=\alpha_{A_1},\cdots,\alpha_{A_{|A|}}$ and $A=\{A_1,\cdots,A_{|A|}\}$,
\begin{equation*}
I=\sum_{\gamma_2\leftrightarrows\gamma_3}\sum_{\begin{subarray}{l}
A\sqcup\,B=S\\A,B\neq\emptyset
\end{subarray}}\textrm{Sign}(\gamma_2,\gamma_3)\LD\gamma_1,\gamma_3,\delta_A,\mu\RD_{0,|A|+3}\eta^{\mu,\nu}\LD\nu,\delta_B,\gamma_2,\gamma_4\RD_{0,n+3-|A|}
\end{equation*}
$\sum_{\gamma_2\leftrightarrows\gamma_3}$ means exchange $\gamma_2, \gamma_3$, and sum up.
$\textrm{Sign}(\gamma_2,\gamma_3)=1, \textrm{Sign}(\gamma_3,\gamma_2)=-1$.
Here we also use the Einstein summation convention for $\mu,\nu$.
According to \cite{FJR2} Lemma 6.2.6 and Lemma 6.2.8, using the above WDVV equation, all genus-0 primary correlators can be reconstructed uniquely from basic correlators. For all the three singularities listed above, the selection rule (\ref{deg:FJRW}) implies the number of marked points for a basic correlator $\LD\alpha_1,\cdots,\alpha_k\RD_{0,k}$ should satisfy
\begin{equation*}
k-2=\sum_{i=1}^k\deg_W\alpha_i\leq(k-2)P+2
\end{equation*}
where for the singularity $W$, $P$ is the maximum complex degree for the corresponding FJRW-primitive class. We can easily compute $P=1/3$ for $P_8$, $J_{10}^T$, and $P=1/4$ for $X_{9}^T.$
Thus, for the $X_{9}^T$ case, $k=4$. For the other two cases, $k=5$. We list all the basic 5-point correlators. Up to symmetry, they are
$\LD\be_{x},\be_{x},\be_{x},\be_{xyz},\be_{xyz}\RD_{0,5}^{P_8}$,
$\LD\be_{x},\be_{x},\be_{y},\be_{xyz},\be_{xyz}\RD_{0,5}^{P_8}$,
$\LD\be_{x},\be_{y},\be_{z},\be_{xyz},\be_{xyz}\RD_{0,5}^{P_8}$
and $\LD\be_8,\be_8,\be_8,\be_{10},\be_{10}\RD_{0,5}^{J_{10}^T}$.

Now we apply the WDVV equation (\ref{eq:WDVV-FJRW}) to $\LD\be_{x},\be_{x},\be_{x},\be_{xyz},\be_{xyz}\RD_{0,5}^{P_8}$. We choose $\gamma_1=\be_x,\gamma_2=\be_{xyz},\gamma_3=\be_{xz},\gamma_4=\be_y,\delta_1=\delta_2=\be_x$, then
$\be_{x}\star\be_{xyz}=\be_x\star\be_{xz}=\be_{xyz}\star\be_y=0$,
and the reconstruction follows. The other three cases are reconstructed similarly.
\end{proof}
Now we give the proof of Theorem \ref{thm:FJRW g=0 reconstruction},
\begin{proof}
We classify all genus-0 4-point \emph{basic} primary correlators.

{\bf For $P_8$ case:}
We list all the possible non-vanishing basic 4-point correlators up to symmetry. They are:
\begin{enumerate}
\item
    $\LD\be_{x},\be_{x},\be_{x},\be_{xyz}\RD_{0}^{P_8},\LD\be_{x},\be_{x},\be_{xy},\be_{xz}\RD_{0}^{P_8};$
\item $\LD\be_{x},\be_{y},\be_{z},\be_{xyz}\RD_{0}^{P_8},\LD\be_{x},\be_{x},\be_{yz},\be_{yz}\RD_{0}^{P_8},\LD\be_{x},\be_{y},\be_{xz},\be_{yz}\RD_{0}^{P_8};$
\item     $\LD\be_{x},\be_{x},\be_{y},\be_{xyz}\RD_{0}^{P_8},\LD\be_{x},\be_{y},\be_{xz},\be_{xz}\RD_{0}^{P_8},\LD\be_{x},\be_{y},\be_{xy},\be_{xz}\RD_{0}^{P_8};$
\item
    $\LD\be_{x},\be_{x},\be_{xy},\be_{xy}\RD_{0}^{P_8},\LD\be_{x},\be_{y},\be_{xy},\be_{xy}\RD_{0}^{P_8}.$
\end{enumerate}
Applying the WDVV equation (\ref{eq:WDVV-FJRW}) over and over again, we can show that all the correlators can be expressed as the scalar multiples of the first one in every row.
For example,
\begin{equation*}
\LD\be_{x},\be_{x},\be_{xy},\be_{xz}\RD_{0}^{P_8}=\LD\be_{x},\be_{x},\be_{x}\star\be_{y},\be_{xz}\RD_{0}^{P_8}=\LD\be_{x},\be_{x},\be_{x},\be_{y}\star\be_{xz}\RD_{0}^{P_8}.
\end{equation*}
Other cases are similar and we leave them to readers as an exercise.
Moreover, the scalar is determined by 3-point correlators which are the initial conditions of our reconstruction.
Furthermore,  we have vanishing results for the last two rows. For example,
as $\be_{x}\star\be_{xy}=\be_{x}\star\be_{x}=\be_{xy}\star\be_{y}=0$, WDVV equation (\ref{eq:WDVV-FJRW}) implies
\begin{equation*}
\LD\be_{x},\be_{x},\be_{xy},\be_{x}\star\be_{y}\RD_{0}^{P_8}=0.
\end{equation*}
Thus we only need to compute $\LD\be_{x},\be_{x},\be_{x},\be_{xyz}\RD_{0}^{P_8},$ and $\LD\be_{x},\be_{y},\be_{z},\be_{xyz}\RD_{0}^{P_8}$.

\begin{rem}
The above vanishing results can also be obtained by the line bundle criterion (\ref{deg:Line BD}). The same applies to $J_{10}^T$ case for  $\LD\be_{1},\be_{8},\be_{5},\be_{10}\RD_{0}^{J_{10}^T}=0$. However, there is no such criterion in the B-model. Here, we stick with the WDVV equation which applies for both the A-model and the B-model.
\end{rem}

{\bf $X_{9}^T$-case:}
There are 18 basic 4-point correlators. Using the WDVV equation,
\begin{equation}\label{eq:FJRW-baisc1}
\LD\be_{1},\alpha,\beta,\be_5\star\gamma\RD_{0}^{X_{9}^T}
+\LD\be_{1}\star\alpha,\beta,\be_5,\gamma\RD_{0}^{X_{9}^T}
=\LD\be_{1},\gamma,\beta,\be_5\star\alpha\RD_{0}^{X_{9}^T}
+\LD\be_{1}\star\gamma,\beta,\be_5,\alpha\RD_{0}^{X_{9}^T}
\end{equation}
\begin{equation}\label{eq:FJRW-basic2}
\LD\be_5,\alpha,\beta,\be_1\star\xi\RD_{0}^{X_{9}^T}
+\LD\be_5\star\alpha,\beta,\be_1,\xi\RD_{0}^{X_{9}^T}
=\LD\be_5,\xi,\beta,\be_1\star\alpha\RD_{0}^{X_{9}^T}
+\LD\be_5\star\xi,\beta,\be_1,\alpha\RD_{0}^{X_{9}^T}
\end{equation}
We can choose as follows:
\begin{itemize}
\item
$\alpha=\be_2,\be_{10},x\be_6$, $\beta=\be_1,\be_5$, $\gamma=\be_7$,  $\xi=\be_{11}$.
\item
$\alpha=\be_8,\beta=\be_1$ or $\be_5, \gamma=\be_5, \xi=\be_1$.
\item
$\alpha=\be_{11},\beta=\be_1,\gamma=\be_{10}$ in case of (\ref{eq:FJRW-baisc1}).
\end{itemize}
There are 17 equations among the 18 basic 4-point correlators. For example, the last choice gives
\begin{equation*}
\LD\be_{1},\be_{11},\be_1,\be_{11}\RD_{0}^{X_{9}^T}+\LD\be_8,\be_1,\be_5,\be_{10}\RD_{0}^{X_{9}^T}
=\LD\be_7,\be_1,\be_5,\be_{11}\RD_{0}^{X_{9}^T}
\end{equation*}
By tedious simplification, we find that all the 18 basic 4-point correlators are scalar multiple of $\LD\be_1,\be_5,\be_7,\be_7\RD_{0}^{X_{9}^T}$.

{\bf For $J_{10}^T$-case:}
We use the same technique. Finally, the basic 4-point correlators are all scalar multiple of the following three special ones:
$\LD\be_{1},\be_{8},\be_{5},\be_{10}\RD_{0}^{J_{10}^T}$,
$\LD\be_{8},\be_{8},\be_{8},\be_{10}\RD_{0}^{J_{10}^T}$, $\LD\be_{1},\be_{1},\be_{4},\be_{10}\RD_{0}^{J_{10}^T}$.

\end{proof}

\subsubsection{Genus-1 reconstruction}
In this subsection, we show that the genus-1 primary correlators can be reconstructed from genus-0 primary correlators.
By the selection rule (\ref{deg:FJRW}), the nonvanishing genus-1 primary correlators must be of the form
$\langle\be_{J^{-1}},\cdots,\be_{J^{-1}}\rangle_{1,n}$, where $\be_{J^{-1}}=\be_{xyz},\be_{8}$ and $\be_{10}$,
respectively in the $P_8, X_9^T$ and $J_{10}^T$ cases.

Our main tool is the Getzler's relation. In \cite{Ge}, Getzler introduced a linear relation between codimension two cycles in $H_*(\overline{\mathscr{M}}_{1,4},\mathbb{Q})$. Here we briefly introduce this relation for our purpose. Consider the dual graph,
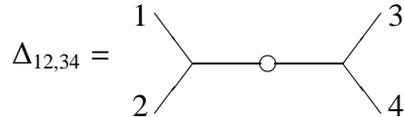
\begin{center}
\begin{picture}(50,20)
    \put(-24,9){$\Delta_{12,34}=$}

    \put(10,9){\circle{2}}

	\put(0,9){\line(-3,4){5}}
    \put(0,9){\line(-3,-4){5}}
	\put(0,9){\line(1,0){9}}
	\put(11,9){\line(1,0){9}}
    \put(20,9){\line(3,4){5}}
    \put(20,9){\line(3,-4){5}}

    \put(-8,2){2}
    \put(-8,14){1}
    \put(26,2){4}
    \put(26,14){3}
\end{picture}
\end{center}
This graph represents a codimension-two stratum in $\overline{\mathscr{M}}_{1,4}$: A circle represents a genus-1 component, other vertices represent genus-0 components. An edge connecting two vertices represents a node, a tail (or half-edge) represents a marked point on the component of the corresponding vertex. $\Delta_{2,2}$ is defined to be the $S_4$-invariant of the codimension-two stratum in $\overline{\mathscr{M}}_{1,4}$,
\begin{equation*}
\Delta_{2,2}=\Delta_{12,34}+\Delta_{13,24}+\Delta_{14,23}
\end{equation*}
We denote $\delta_{2,2}=[\Delta_{2,2}]$ the corresponding cycle in $H_4(\overline{\mathscr{M}}_{1,4},\mathbb{Q})$. Other strata are defined similarly. For more details, see \cite{Ge}. We list the corresponding dual graph,
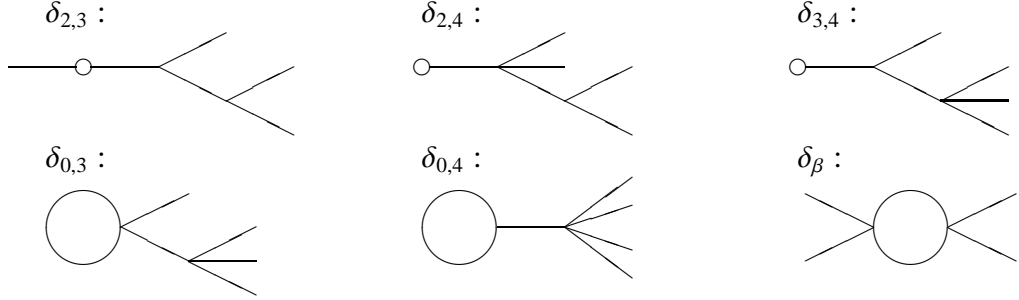
\begin{center}
\begin{picture}(50,20)

    \put(-35,9){\circle{2}}
	\put(-40,15){$\delta_{2,3}:$}

	\put(-45,9){\line(1,0){9}}
	\put(-34,9){\line(1,0){9}}
    \put(-25,9){\line(2,1){9}}
    \put(-25,9){\line(2,-1){9}}
    \put(-16,4.5){\line(2,1){9}}
    \put(-16,4.5){\line(2,-1){9}}

    \put(10,9){\circle{2}}
	\put(10,15){$\delta_{2,4}:$}


	\put(11,9){\line(1,0){9}}
    \put(20,9){\line(2,1){9}}
	\put(20,9){\line(1,0){9}}
    \put(20,9){\line(2,-1){9}}
    \put(29,4.5){\line(2,1){9}}
    \put(29,4.5){\line(2,-1){9}}

    \put(60,9){\circle{2}}
	\put(60,15){$\delta_{3,4}:$}

	\put(79,4.5){\line(1,0){9}}
	\put(61,9){\line(1,0){9}}
    \put(70,9){\line(2,1){9}}
    \put(70,9){\line(2,-1){9}}
    \put(79,4.5){\line(2,1){9}}
    \put(79,4.5){\line(2,-1){9}}
\end{picture}
\end{center}

\begin{center}
\begin{picture}(50,20)

	\put(-40,17){$\delta_{0,3}:$}
	\put(10,17){$\delta_{0,4}:$}
	\put(60,17){$\delta_{\beta}:$}

    \put(-35,9){\circle{10}}


	\put(-21,4.5){\line(1,0){9}}
    \put(-30,9){\line(2,1){9}}
    \put(-30,9){\line(2,-1){9}}
    \put(-21,4.5){\line(2,1){9}}
    \put(-21,4.5){\line(2,-1){9}}

    \put(15,9){\circle{10}}


    \put(29,9){\line(4,3){9}}
    \put(29,9){\line(3,1){9}}
	\put(20,9){\line(1,0){9}}
    \put(29,9){\line(4,-3){9}}
    \put(29,9){\line(3,-1){9}}


    \put(75,9){\circle{10}}

    \put(80,9){\line(2,1){9}}
    \put(80,9){\line(2,-1){9}}
    \put(70,9){\line(-2,1){9}}
    \put(70,9){\line(-2,-1){9}}
\end{picture}
\end{center}
According to \cite{Ge}, Getzler's relation is as follows:
\begin{equation}\label{eq:Getzler}
  12\delta_{2,2}+4\delta_{2,3}-2\delta_{2,4}+6\delta_{3,4}+\delta_{0,3}+\delta_{0,4}-2\delta_{\beta}=0.
\end{equation}

\begin{thm}\label{thm:FJRW genus-1}
For all simple elliptic singularites with maximal admissible group, the genus-1 FJRW correlators can be reconstructed from genus-0 FJRW correlators by the Getzler relation.
\end{thm}
\begin{proof}

{\bf $P_8$-case:}

In this case, we need to reconstruct $\LD\be_{xyz},\cdots,\be_{xyz}\RD_{1,n}^{P_8}, n\geq 2.$
We have the forgetful map $\pi_{4,n-2}:\overline{\mathscr{M}}_{1,n+2}\rightarrow\overline{\mathscr{M}}_{1,4}$.
Let $S=\{1,\cdots,n-2\}$. Thus
\begin{equation*}
\pi_{4,n-2}^{-1}(\Delta_{12,34})=\sum_{A\cup\,B\cup\,C=S}\Delta_{12A,B,C34}
\end{equation*}
Now we choose $n+2$ insertions: the first four are $\be_{x},\be_{yz},\be_{y},\be_{xz}$, the others are $\be_{xyz}$.
Integrating the class $\Lambda_{1,n+2}^{P_8}(\be_{x},\be_{yz},\be_{y},\be_{xz},\be_{xyz},\cdots,\be_{xyz})$ on $\pi_{4,n-2}^{-1}([\Delta_{12,34}])$, we have:
\begin{eqnarray*}
\begin{split}
&\int_{\pi_{4,n-2}^{-1}([\Delta_{12,34}])}\Lambda_{1,n+2}^{P_8}(\be_{x},\be_{yz},\be_{y},\be_{xz},\be_{xyz},\cdots,\be_{xyz})\\
&=\int_{[\Delta_{12,S,34}]}\Lambda_{1,n+2}^{P_8}(\be_{x},\be_{yz},\be_{y},\be_{xz},\be_{xyz},\cdots,\be_{xyz})\\
&=\LD\be_x,\be_{yz},1\RD_{0,3}^{P_8}\eta^{1,\be_{xyz}}\LD\be_{xyz},\cdots,\be_{xyz}\RD_{1,n}^{P_8}\eta^{\be_{xyz},1}\LD1,\be_y,\be_{xz}\RD_{0,3}^{P_8}\\
&=\LD\be_{xyz},\cdots,\be_{xyz}\RD_{1,n}^{P_8}
\end{split}
\end{eqnarray*}
The second equality uses the Splitting Axiom.
The first equality is a consequence of the Selection Rule (\ref{deg:FJRW}) and the String equation. The Selection rule requires that each insertion for a non-zero genus-1 primary correlator is $\be_{xyz}$ and the String Equation implies that the genus-0 primary correlator with more than four insertions will vanish if there is one insertion of the identity element. Thus the non-zero contribution partition should be $B=S$, and $A=C=\emptyset$. The corresponding decorated dual graph for $\pi^{-1}_{4,n-2}(\Delta_{12,34})$ is
\begin{center}
\begin{picture}(50,23)

    \put(-30,9){$\Delta_{12,S,34}(\be_{x},\be_{yz},\be_{y},\be_{xz},\be_{xyz},\cdots,\be_{xyz})=$}
    \put(50,9){\circle{2}}

	\put(40,9){\line(-3,4){5}}
    \put(40,9){\line(-3,-4){5}}
	\put(40,9){\line(1,0){9}}
	\put(51,9){\line(1,0){9}}
    \put(60,9){\line(3,4){5}}
    \put(60,9){\line(3,-4){5}}
    \put(50,10){\line(-3,4){4}}
    \put(50,10){\line(3,4){4}}
	
	\put(34,17){$\be_{x}$}
	\put(34,0){$\be_{yz}$}
	\put(64,17){$\be_{y}$}
	\put(64,0){$\be_{xz}$}

	\put(42,17){$\be_{xyz}$}
	\put(48,17){$\cdots$}
	\put(54,17){$\be_{xyz}$}
	\put(48,14){$\cdots$}

\end{picture}
\end{center}
Other decorated dual graphs are obtained similarly. As $\delta_{2,2}=[\Delta_{2,2}]$ is the $S_4$-invariant, we integrate over each stratum and finally get
\begin{equation}\label{eq:FJRW-Getzler1}
\int_{\pi_{4,n-2}^{-1}(\delta_{2,2})}
\Lambda_{1,n+2}^{P_8}(\be_{x},\be_{yz},\be_{y},\be_{xz},\be_{xyz},\cdots,\be_{xyz})
=3\LD\be_{xyz},\cdots,\be_{xyz}\RD_{1,n}^{P_8}
\end{equation}
We observe that only $\delta_{2,3}$ can contain at most $n$ insertions for the genus-1 component. However, one of the insertions is decorated with an element from the first four insertions. Thus the integration vanishes according to the selection rule. On the other hand, when we integrate the same class on other dimension two strata in Getzler's relation (\ref{eq:Getzler}), all the genus-1 correlators will have at most $n-1$ insertions.
Thus Getzler's relation implies $\LD\be_{xyz},\cdots,\be_{xyz}\RD_{1,n}^{P_8}$ $(n\geq2)$ can be reconstructed from genus-1 correlators with fewer insertions and other genus-0 correlators.

Now we consider the integration of the class $\Lambda_{1,4}^{P_8}(\be_{x},\be_x,\be_x,\be_{xyz})$ on those codimension two strata of $\overline{\mathscr{M}}_{1,4}$. We can discuss similarly as above. The integration on $\delta_{2,2}, \delta_{2,3}, \delta_{2,4}$ will all vanish. However,
\begin{eqnarray*}
\begin{split}
\int_{[\Delta_{1,234}]}\Lambda_{1,4}^{P_8}(\be_{x},\be_x,\be_x,\be_{xyz})
&=\LD\be_{xyz}\RD_{1,1}^{P_8}\eta^{\be_{xyz},1}\LD\be_{1},\be_{x},\be_{yz}\RD_{0,3}^{P_8}
\eta^{\be_{yz},\be_x}\LD\be_{x},\be_x,\be_x,\be_{xyz}\RD_{0,4}^{P_8}\\
&=\frac{1}{3}\LD\be_{xyz}\RD_{1,1}^{P_8}
\end{split}
\end{eqnarray*}
Here we use $\LD\be_{x},\be_x,\be_x,\be_{xyz}\RD_{0,4}^{P_8}=\frac{1}{3}$, which is computed in Section 3.3.1. Overall,  $$\int_{\delta_{3,4}}\Lambda_{1,4}^{P_8}(\be_{x},\be_x,\be_x,\be_{xyz})=\frac{4}{3}\LD\be_{xyz}\RD_{1,1}^{P_8}.$$
Now applying the Getzler's relation again, other contributions are of genus-0, and $\LD\be_{xyz}\RD_{1,1}^{P_8}$ can be reconstructed from genus-0 primary correlators.

{\bf $X_9^T$-case:}

For $n\geq 2,$ we choose $n+2$ insertions: the first four are $\be_{1},\be_{11},\be_{5},\be_{7}$, the others are $\be_{8}$. The nonzero contribution of $\Delta_{12,34}$ comes from the following decorated dual graph:
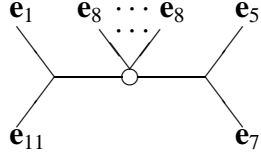
\begin{center}
\begin{picture}(50,23)
    \put(10,9){\circle{2}}

	\put(0,9){\line(-3,4){5}}
    \put(0,9){\line(-3,-4){5}}
	\put(0,9){\line(1,0){9}}
	\put(11,9){\line(1,0){9}}
    \put(20,9){\line(3,4){5}}
    \put(20,9){\line(3,-4){5}}
    \put(10,10){\line(-3,4){4}}
    \put(10,10){\line(3,4){4}}
	
	\put(-6,17){$\be_{1}$}
	\put(-6,0){$\be_{11}$}
	\put(24,17){$\be_{5}$}
	\put(24,0){$\be_{7}$}

	\put(3,17){$\be_{8}$}
	\put(8,17){$\cdots$}
	\put(14,17){$\be_{8}$}

	\put(8,14){$\cdots$}

\end{picture}
\end{center}
and
\begin{equation}\label{eq:FJRW-Getzler2}
\int_{\pi_{4,n-2}^{-1}(\delta_{2,2})}\Lambda_{1,n+2}^{X_9^T}(\be_{1},\be_{11},\be_{5},\be_{7},\be_{8},\cdots,\be_{8})
=3\LD\be_{8},\cdots,\be_{8}\RD_{1,n}^{X_9^T}
\end{equation}
The integrations of $\Lambda_{1,n+2}^{X_9^T}(\be_{1},\be_{11},\be_{5},\be_{7},\be_{8},\cdots,\be_{8})$ on other strata in Getzler's relation only produce genus-1 correlators with lower insertions and genus-0 correlators.

For $n=1$, we integrate $\Lambda_{1,4}^{X_9^T}(\be_{1},\be_{5},\be_{7},\be_{7})$ on the Getzler's relation. It vanishes on strata with genus-1 component except for $\delta_{3,4}$. Using (\ref{eq:n=0,X9}), we have
\begin{equation*}
\int_{\delta_{3,4}}\Lambda_{1,4}^{X_9^T}(\be_{1},\be_{5},\be_{7},\be_{7})
=-\frac{2}{3}\LD\be_{8}\RD_{1,1}^{X_9^T}
\end{equation*}
Thus reconstruction of the genus-1 primary correlators follows.

{\bf $J_{10}^T$-case:}

For $n\geq2$, we integrate the class $\Lambda_{1,n+2}^{J_{10}^T}(\be_{1},\be_{11},\be_{8},\be_{4},\be_{10},\cdots,\be_{10})$ over the Getzler's relation. The non-zero contribution of integrating over $\delta_{2,2}$ comes from three decorated dual graphs. One of them is
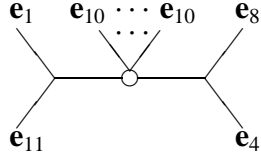
\begin{center}
\begin{picture}(50,23)
    \put(10,9){\circle{2}}

	\put(0,9){\line(-3,4){5}}
    \put(0,9){\line(-3,-4){5}}
	\put(0,9){\line(1,0){9}}
	\put(11,9){\line(1,0){9}}
    \put(20,9){\line(3,4){5}}
    \put(20,9){\line(3,-4){5}}
    \put(10,10){\line(-3,4){4}}
    \put(10,10){\line(3,4){4}}
	
	\put(-6,17){$\be_{1}$}
	\put(-6,0){$\be_{11}$}
	\put(24,17){$\be_{8}$}
	\put(24,0){$\be_{4}$}

	\put(2,17){$\be_{10}$}
	\put(8,17){$\cdots$}
	\put(14,17){$\be_{10}$}

	\put(8,14){$\cdots$}

\end{picture}
\end{center}
Overall, we have
\begin{equation}\label{eq:FJRW-Getzler3}
\int_{\pi_{4,n-2}^{-1}(\delta_{2,2})}\Lambda_{1,n+2}^{J_{10}^T}(\be_{1},\be_{11},\be_{8},\be_{4},\be_{10},\cdots,\be_{10})
=3\LD\be_{10},\cdots,\be_{10}\RD_{1,n}^{J_{10}^T}
\end{equation}
Then integrations of $\Lambda_{1,n+2}^{J_{10}^T}(\be_{1},\be_{11},\be_{8},\be_{4},\be_{10},\cdots,\be_{10})$ on other strata in Getzler's relation only produce genus-1 correlators with lower insertions and genus-0 correlators. Thus the reconstruction follows for $n\geq2$.

For $n=1$, we integrate the class $\Lambda_{1,4}^{J_{10}^T}(\be_{8},\be_{8},\be_{8},\be_{10})$.
The unique genus-1 correlators contribution comes from $\delta_{3,4}$. Using (\ref{eq:n=0,J10}), we have
\begin{equation*}
\int_{\delta_{3,4}}\Lambda_{1,4}^{J_{10}^T}(\be_{8},\be_{8},\be_{8},\be_{10})
=\frac{4}{3}\LD\be_{10}\RD_{1,1}^{J_{10}^T}
\end{equation*}
All the other contributions are of genus-0 correlators. Thus the reconstruction holds.
\end{proof}

\subsubsection{Higher genus reconstruction}
Now let us prove the reconstruction theorem of all genera for three three types of singularities paired with $G_{\textrm{max}}$. The key point is the \emph{$g$-reduction}. As we need the explicit form in the next subsection, we reproduce here. The \emph{$g$-reduction} lemma is
\begin{lm}\label{lm:g-reduction}
Let $\beta$ be a monomial in the $\psi$ and $\kappa$-classes in $\overline{\mathscr{M}}_{g,n}$ of degree at least $g$ for $g\geq1$ or at least 1 for $g=0$. Then the class $\beta$ can be represented by a linear combination of dual graphs, each of which has at least one edge.
\end{lm}
It was first used in \cite{FSZ} for proving Witten's conjecture for $r$-spin curves. Then in \cite{FJR2}, it was generalized to case of central charge $\hat{c}_W\leq 1$ in the setting of FJRW theory, which includes the $r$-spin case as type $A_{r-1}$ singularities $W:=x^r$.
\begin{lm}
For the three types of elliptic singularities $W=P_8, X_{9}^T, J_{10}^T$, the correlator $\LD\tau_{l_1}(\alpha_1),\cdots,\tau_{l_n}(\alpha_n)\RD_{g,n}^{W,G}$ in FJRW theory is uniquely reconstructed by tautological relations (WDVV-equation, Getzler's relation and $g$-reduction) from genus-0 primary correlators.
\end{lm}
\begin{proof}
As for these three elliptic singularities, the central charge $\hat{c}_W=1$, thus the result easily follows from Theorem \ref{thm:FJRW genus-1} here and Theorem 6.2.1 in \cite{FJR2}.
\end{proof}

\subsection{LG-to-LG Mirror theorem}



In this section, we establish the LG-to-LG mirror theorem for all genera.

\begin{thm}\label{thm:LG-LG mirror}
For $W:=P_{8},X_{9},J_{10}$, we can choose the coordinates appropriately such that the ancestor potential
\begin{equation}
\cal{A}_{FJRW}^{W^T,G_{\textrm{max}}}(\textbf{t}_A)=\cal{A}_{W}^{SG}(\textbf{t}_B).
\end{equation}
where $\textbf{t}_A, \textbf{t}_B$ are semi-simple points.
\end{thm}

First of all, we have identical reconstruction theorem of Saito-Givental theory as that of FJRW-theory.

\begin{thm}\label{thm:FJRW-reconstruction}
For $W=P_8, X_{9}, J_{10}$, the correlator $\LD\tau_{l_1}(\alpha_1),\cdots,\tau_{l_n}(\alpha_n)\RD_{g,n}^{W}$ in Saito-Givental theory is uniquely reconstructed by tautological relations from the pairing, Milnor ring structural constants and genus-0 4-point correlators $\LD\,x,x,x,xyz\RD_{0}^{P_{8}}$, $\LD\,x,y,z,xyz\RD_{0}^{P_{8}}$, $\LD\,x,y,xz,xz\RD_{0}^{X_{9}}$,
 $\LD\,y,y,y,xyz\RD_{0}^{J_{10}}$, $\LD\,x,x,xz,xyz\RD_{0}^{J_{10}}$ and $\LD\,x,y,z,xyz\RD_{0}^{J_{10}}$, respectively.
\end{thm}

To prove the theorem, we only need to match 3-point correlators and the above 4-point correlators. The 3-point correlators have been matched already by the first author. In this subsection, we consider those 4-point correlators.

\subsubsection{Computation for genus-0 4-point correlators of FJRW-theory}
As we already pointed out, the line bundle criterion (\ref{deg:Line BD}) implies that $\LD\be_x,\be_y,\be_z,\be_{xyz}\RD_{0}^{P_{8}}$ and $\LD\be_1,\be_8,\be_5,\be_{10}\RD_{0}^{J_{10}^T}$ vanish. For the other cases, we use the orbifold Grothendieck-Riemann-Roch formula discussed in \cite{FJR2}.
\begin{thm}
Assume the $W$-structure is concave with all markings are narrow. If the markings are labeled with group elements $\gamma_i$ and the complex codimension $D=1$, then the class $\Lambda_{g,k}^{W,G}(\mathbf{e}_{\gamma_1},\cdots,\mathbf{e}_{\gamma_{k}})\in H^*(\overline{\mathscr{M}}_{g,k},\mathbb{C})$ is giving by the following:
\begin{eqnarray}\label{eq:OGRR}
\begin{split}
&&\Lambda_{g,k}^{W,G}(\mathbf{e}_{\gamma_1},\cdots,\mathbf{e}_{\gamma_{k}})
=
\sum_{l=1}^N\Big(
(\frac{q_l^2}{2}-\frac{q_l}{2}+\frac{1}{12})\kappa_1
-\sum_{i=1}^k(\frac{1}{12}-\frac{1}{2}\Theta_{l}^{\gamma_i}(1-\Theta_{l}^{\gamma_i}))\psi_i
\\
&&+\sum_{\Gamma\in\mathbf{\Gamma}_{g,k,W}(\gamma_1,\cdots,\gamma_k)}
(\frac{1}{12}-\frac{1}{2}\Theta_{l}^{\gamma_{\Gamma}}(1-\Theta_{l}^{\gamma_{\Gamma}}))[\overline{\mathscr{M}}(|\Gamma|)]
\Big)
\end{split}
\end{eqnarray}
where $\mathbf{\Gamma}_{g,k,W}(\gamma_1,\cdots,\gamma_k)$ is the set of single-edged $G$-decorated dual graphs for FJRW-theory.
\end{thm}

{\bf $P_8=x^3+y^3+z^3$-Case:}

According to Theorem \ref{thm:FJRW g=0 reconstruction}, we only need to compute $\LD\be_{x},\be_{x},\be_{x},\be_{xyz}\RD_{0,4}^{P_8}$. This correlator satisfies the concavity condition, thus we can use the theorem above. There are three decorated dual graphs in $\mathbf{\Gamma}_{0,4,P_8}(\gamma_{x},\gamma_{x},\gamma_{x},\gamma_{xyz})$. Each of them is:
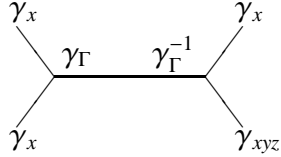
\begin{center}
\begin{picture}(50,20)

	\put(0,9){\line(-3,4){5}}
    \put(0,9){\line(-3,-4){5}}
	\put(0,9){\line(1,0){20}}
    \put(20,9){\line(3,4){5}}
    \put(20,9){\line(3,-4){5}}
	
	\put(-6,17){$\gamma_x$}
	\put(-6,0){$\gamma_x$}
	\put(24,17){$\gamma_x$}
	\put(24,0){$\gamma_{xyz}$}
	\put(1,11){$\gamma_{\Gamma}$}
	\put(13,11){$\gamma_{\Gamma}^{-1}$}
\end{picture}
\end{center}
We denote each by $\Gamma$ and count the contribution of $\Gamma$ three times. Here the edge is labeled by $\gamma_{\Gamma}=\textrm{Diag}(1,\omega^{2},\omega^{2})\in G_{\textrm{max}}$. Thus $\Theta_{x}^{\gamma_{\Gamma}}=0.$

Further more, we have $\deg|\cal{L}_x|=-2, \deg|\cal{L}_y|=-1, \deg|\cal{L}_z|=-1.$
Thus we have $R^1\pi_*|\cal{L}_y|=R^1\pi_*|\cal{L}_z|=0$, and the only contribution is from $|\cal{L}_x|$.
Using the fact that
\begin{equation*}
\int_{\overline{\mathscr{M}}_{0,4}}\kappa_1=\int_{\overline{\mathscr{M}}_{0,4}}\psi_i
=\int_{\overline{\mathscr{M}}_{0,4}}[\overline{\mathscr{M}}(|\Gamma|)]
=1,
\end{equation*}
we  have
\begin{eqnarray*}
\begin{split}
&\LD\be_{x},\be_{x},\be_{x},\be_{xyz}\RD_{0,4}^{P_8}
=\int_{\overline{\mathscr{M}}_{0,4}}\Lambda_{0,4}^{P_8}(\be_{x},\be_{x},\be_{x},\be_{xyz})\\
&=(\frac{q_x^2}{2}-\frac{q_x}{2}+\frac{1}{12})
-\sum_{i=1}^{4}(\frac{1}{12}-\frac{1}{2}\Theta_{x}^{\gamma_i}(1-\Theta_{x}^{\gamma_i}))
+3(\frac{1}{12}-\frac{1}{2}\Theta_{x}^{\gamma_{\Gamma}}(1-\Theta_{x}^{\gamma_{\Gamma}}))\\
&=\frac{1}{3}
\end{split}
\end{eqnarray*}


{\bf $X_{9}^T=x^2y+y^3+xz^2$-Case: }

Here we compute $\LD\be_{1},\be_{5},\be_{7},\be_{7}\RD_{0,4}^{X_9^T}$. It is also concave, and we have $\deg|\cal{L}_x|=-2$, $\deg|\cal{L}_y|=-1,$ $\deg|\cal{L}_z|=-1.$ There are three dual graphs in $\mathbf{\Gamma}_{0,4,X_{9}^T}(\gamma^1,\gamma^5,\gamma^7,\gamma^7)$. Two of them are as follows, with one of the half edge is decorated by $\gamma_{\Gamma_1}=\gamma^8$:
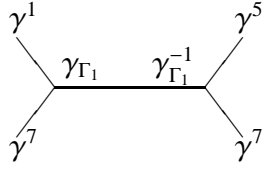
\begin{center}
\begin{picture}(50,20)
	\put(0,9){\line(-3,4){5}}
    \put(0,9){\line(-3,-4){5}}
	\put(0,9){\line(1,0){20}}
    \put(20,9){\line(3,4){5}}
    \put(20,9){\line(3,-4){5}}
	
	\put(-6,17){$\gamma^1$}
	\put(-6,0){$\gamma^7$}
	\put(24,17){$\gamma^5$}
	\put(24,0){$\gamma^7$}
	\put(1,11){$\gamma_{\Gamma_1}$}
	\put(13,11){$\gamma_{\Gamma_1}^{-1}$}
\end{picture}
\end{center}
The third dual graph is as follows, with one of the half edge decorated with $\gamma_{\Gamma_2}=\gamma^2$:
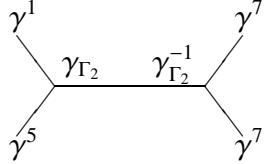
\begin{center}
\begin{picture}(50,20)
	\put(0,9){\line(-3,4){5}}
    \put(0,9){\line(-3,-4){5}}
	\put(0,9){\line(1,0){20}}
    \put(20,9){\line(3,4){5}}
    \put(20,9){\line(3,-4){5}}
	
	\put(-6,17){$\gamma^1$}
	\put(-6,0){$\gamma^5$}
	\put(24,17){$\gamma^7$}
	\put(24,0){$\gamma^7$}
	\put(1,11){$\gamma_{\Gamma_2}$}
	\put(13,11){$\gamma_{\Gamma_2}^{-1}$}
\end{picture}
\end{center}
Those graphs give $\Theta_x^{\gamma_{\Gamma_1}}=\frac{2}{3}, \Theta_x^{\gamma_{\Gamma_2}}=\frac{1}{3}.$
Thus we get
\begin{equation}\label{eq:n=0,X9}
\LD\be_{1},\be_{5},\be_7,\be_7\RD_{0,4}=-\frac{1}{6}.
\end{equation}

{\bf $J_{10}^T=x^3+y^3+xz^2$-Case:}

As for the correlator $\LD\be_{8},\be_{8},\be_{8},\be_{10}\RD_{0,4}^{J_{10}^T}$, we have $\deg|\cal{L}_x|=-1$, $\deg|\cal{L}_y|=-2$ and $\deg|\cal{L}_z|=-1.$ There are three decorated dual graphs, with $\gamma_{\Gamma}=(1,\omega^{2},\omega^2)\in G_{\textrm{max}}$.
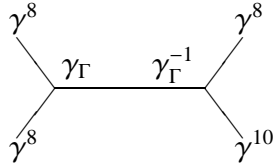
\begin{center}
\begin{picture}(50,20)
	\put(0,9){\line(-3,4){5}}
    \put(0,9){\line(-3,-4){5}}
	\put(0,9){\line(1,0){20}}
    \put(20,9){\line(3,4){5}}
    \put(20,9){\line(3,-4){5}}

	
	\put(-6,17){$\gamma^8$}
	\put(-6,0){$\gamma^8$}
	\put(24,17){$\gamma^{8}$}
	\put(24,0){$\gamma^{10}$}
	\put(1,11){$\gamma_{\Gamma}$}
	\put(13,11){$\gamma_{\Gamma}^{-1}$}
\end{picture}
\end{center}
Applying the orbifold Riemann-Roch formula (\ref{eq:OGRR}) again, we have
\begin{equation}\label{eq:n=0,J10}
\LD\be_{8},\be_{8},\be_{8},\be_{10}\RD_{0,4}^{J_{10}^T}=\frac{1}{3}.
\end{equation}
Similarly, we can compute $$\LD\be_{1},\be_{1},\be_{4},\be_{10}\RD_{0,4}^{J_{10}^T}=\frac{1}{3}.$$


\subsubsection{Computation for genus-0 4-point correlators of Saito-Givental theory}
Here we compute the basic genus-0 4-point correlators of $P_8, X_9, J_{10}$ in Saito-Givental theory. Our methods follow from Noumi, \cite{No}. For $P_8$ case, we refer to the formula directly from \cite{NY}.

{\bf $P_8$-Case:}

The miniversal deformation of $P_8$ is given by
\begin{equation*}
P_{8}(\textbf{s},\textbf{x})=x^3+y^3+z^3+\sum_{i=-1}^{6}s_i\phi_i.
\end{equation*}
where the $\phi_i$'s form a basis as follows:
\begin{equation*}
\phi_{-1}=xyz,\phi_0=1,\phi_1=x,\phi_2=y,\phi_3=z,\phi_4=xy,\phi_5=xz,\phi_6=yz.
\end{equation*}
Using Noumi-Yamada's formula in \cite{NY}, we know
\begin{equation}
\pi_{A}(s_{-1})=\ _2F_1\Big(\frac{1}{3},\frac{1}{3};\frac{2}{3};-\frac{s_{-1}^3}{27}\Big)
\end{equation}
Now the flat coordinates $t_i$ are expanded in the coordinates $s_i$ up to second order:
\begin{equation*}
t_i=s_i+o(s^2).
\end{equation*}
Here $o(s^2)$ is a linear expansion of $s_i$'s of order greater than 2.
Notice here we have $s_{-1}=t_{-1}$. The deformed Jacobi ideal is generated by:
\begin{eqnarray*}
\partial_{x}P_{8}(\textbf{s},\textbf{x})&=&3x^2+s_{-1}yz+s_{4}y+s_{5}z+s_{1}\\
\partial_{y}P_{8}(\textbf{s},\textbf{x})&=&3y^2+s_{-1}xz+s_{4}x+s_{6}z+s_{2}\\
\partial_{z}P_{8}(\textbf{s},\textbf{x})&=&3z^2+s_{-1}xy+s_{5}x+s_{6}y+s_{3}
\end{eqnarray*}
and
$$\textrm{Hess}(P_{8}(\textbf{s},\textbf{x}))=(216+8s_{-1}^3)xyz\mod\textrm{Jac}(P_{8}).$$

The B-model genus-0 correlators with three or four insertions are:
\begin{eqnarray}
C_{i,j,k}(\textbf{t}):&=&\textrm{Res}_{\textbf{t}}\frac{W_{t_i}W_{t_j}W_{t_k}\,dx\wedge\,dy\wedge\,dz}{W_x\,W_y\,W_z}=C\mu
\\
C_{i,j,k,l}(\textbf{t}):&=&\frac{\partial C_{i,j,k}}{\partial t_l}
\end{eqnarray}
$C$ is the unique constant such that $W_{t_i}W_{t_j}W_{t_k}=C\cdot \textrm{Hessian}(W)\mod\,\textrm{Jac}(W)$.

Thus we have the following genus-0 4-point correlators:
\begin{eqnarray*}
\begin{split}
&\LD x,y,z,xyz\RD_{0,4}^{P_8}=C_{1,2,3,-1}(0)=0,\\
&\LD x,x,x,xyz\RD_{0,4}^{P_8}=C_{1,1,1,-1}(0)=\frac{\partial(x^3)}{\partial t_{-1}}|_{\textbf{t}=0}=-\frac{1}{3}.
\end{split}
\end{eqnarray*}
We define $\bar\Psi_{P_8}:\phi_i\mapsto(-1)^{1-d_i}\Psi_{P_8}(\phi_i)$ and rescale the primitive form by $-1$; then $\bar\Psi_{P_8}$ matches the ring structural constants and basic genus-0 four-point correlators.
\begin{equation}
\bar\Psi_{P_8}\Big(F_3^B(P_8)\Big)=F_3^A(P_8), \bar\Psi_{P_8}\Big(F_4^B(P_8)\Big)=F_4^A(P_8).
\end{equation}
By Theorem \ref{thm:FJRW-reconstruction}, this proves the LG-to-LG mirror theorem for $P_8$.

{\bf $X_9$-Case: }

The miniversal deformation is
\begin{equation*}
X_9(\textbf{s},\textbf{x}):=x^2z+xy^3+z^2+\sum_{i=-1}^{7}s_i\phi_i
\end{equation*}
where $\phi_{-1}=xyz;\phi_0=1;\phi_1=x;\phi_2=y;\phi_3=z;\phi_4=y^2;\phi_5=xy;\phi_6=xz;\phi_7=yz.$
The Picard-Fuchs equation is the same as in the $P_8$ case. Consider the period matrix, we denote a multi-valued flat section of the cohomology bundle
\begin{equation*}
\Phi_{i_1,\cdots,i_k}:=\int\phi_{i_1}\cdots\phi_{i_k}\frac{dxdydz}{df}.
\end{equation*}
The Gauss-Manin connection (\ref{eq:Guass-Manin}) implies $\Phi_{i_1,\cdots,i_k}$ satisfies a system of differential equations. For example, $\Phi_1$ and $\Phi_2$ satisfy
\begin{eqnarray}
\begin{split}
&(1+\frac{s_{-1}^3}{27})\Phi_1'=-\frac{s_{-1}^2}{108}\Phi_1-\frac{1}{6}\Phi_2\\
&(1+\frac{s_{-1}^3}{27})\Phi_2'=\frac{s_{-1}}{36}\Phi_1-\frac{s_{-1}^2}{54}\Phi_2\\
\end{split}
\end{eqnarray}
Thus $\Phi_1$ will satisfy a hypergeometric equation and can be solved:
\begin{equation*}
\Phi_1=
\,_2F_1(\frac{1}{12},\frac{5}{6};\frac{2}{3};-\frac{s_{-1}^3}{27})A_1
-\frac{s_{-1}}{6}\,_2F_1(\frac{5}{12},\frac{7}{6};\frac{4}{3};-\frac{s_{-1}^3}{27})A_2
\end{equation*}
where $A_i$ is some flat section with eigenvalues $e^{2\pi i (-d_i)}$ under the monodromy operator. For our purpose, we solve all $\Phi_i$'s to the first order in $s_{-1}$,
\begin{align*}
\Phi_1&=A_1-s_{-1}{A_2}/6+o(s_{-1},A), & \Phi_2&=A_2+o(s_{-1},A), & \Phi_3&=A_3-s_{-1}{A_5}/6+o(s_{-1},A),\\
\Phi_4&=-s_{-1}{A_3}/3+A_4+o(s_{-1},A), & \Phi_5&=A_5+o(s_{-1},A), & \Phi_6&=A_6-s_{-1}{A_7}/6+o(s_{-1},A),\\
\Phi_7&=A_7+o(s_{-1},A).&&&&
\end{align*}
Here $o(s_{-1},A)$ is the $H^2(X_{s,\lambda})$-valued linear expansion of $s_{-1}$ of order greater than 1.
Take the second-order expansion of $C_{\delta}(\textbf{s},\textbf{x})$, and then integrate, for example:
\begin{equation*}
C_{\frac{1}{4}}(\textbf{s},\textbf{x})=\widetilde{\Gamma}(\frac{3}{4})\Big(s_6 \phi_6+s_7 \phi_7+o(s^2,\phi)\Big)
\end{equation*}
Here $o(s^2,\phi)$ means $\mathscr{Q}_W$-valued linear expansions of $s_i$'s of order greater than 2.
We can choose $\alpha_i\in\,H^2(X_{\sigma,1},\mathbb{C})$ such that the intersection pairing is
\begin{equation*}
\LD\alpha_i,A_j\RD=\int_{\alpha_i}A_j=\widetilde{\Gamma}(\frac{3}{4})^{-1}\delta_{ij}.
\end{equation*}
The choice gives the expression of the period matrix as in (\ref{eq:flat coord}). Thus,
\begin{align*}
t_6=\Pi_{\frac{1}{2},6},& & t_7=\Pi_{\frac{1}{4},7}.
\end{align*}
where $t_6$ and $t_7$ are flat coordinates. Now we have:
\begin{equation*}
t_7=\int_{\alpha_7}C_{\frac{1}{4}}(\textbf{s},\textbf{x})\frac{\omega}{df}=\frac{\widetilde{\Gamma}(\frac{3}{4})}{\pi_A(s_{-1})}\int_{\alpha_7}\Big(s_6 \phi_6+s_7 \phi_7+o(s^2,\phi)\Big)\frac{dxdydz}{df}
=s_7-\frac{1}{6}s_{-1}s_6+o(s^2)
\end{equation*}
\begin{rem}
Here we also have to compute the expansion of $\Phi_{i,j}$.
\end{rem}
Besides $t_{-1}=s_{-1}$, other flat coordinates have the second-order expansion
\begin{align*}
t_0 & = s_0-{s_1s_6}/3-{s_4s_5}/3-{s_3^2}/3 +o(s^2),&
t_1 & = s_1-{s_3s_6}/12+{s_4s_7}/6 +o(s^2),\\
t_2 & = s_2-{s_{-1}s_1}/6-{s_5s_6}/6-{s_3s_7}/3 +o(s^2),&
t_3 & = s_3-{s_{-1}s_4}/3-{s_6^2}/6 +o(s^2),\\
t_4 & = s_4-{s_7^2}/4 +o(s^2),&
t_5 & = s_5-{s_{-1}s_3}/6-{s_6s_7}/6 +o(s^2),\\
t_6 & = s_6 +o(s^2),&
t_7 & = s_7-{s_{-1}s_6}/6 +o(s^2).
\end{align*}
Thus we have
\begin{eqnarray*}
\begin{split}
&\LD x,y,xz,xz\RD_{0,4}=C_{1,2,6,6}(0)=\frac{1}{6}.
\end{split}
\end{eqnarray*}
We define $\bar\Psi_{X_9}:\phi_i\mapsto(-1)^{1-d_i}\Psi_{X_9}(\phi_i)$ and rescale the primitive form by $-1$; then $\bar\Psi_{X_9}$ matches the ring structural constants and basic genus-0 4-point correlators. By Theorem \ref{thm:FJRW-reconstruction}, this proves the LG-to-LG mirror theorem for $X_9$.


{\bf $J_{10}$-Case: }

For this case, the miniversal deformation is
\begin{equation*}
J_{10}(\textbf{s},\textbf{x})=x^3z+y^3+z^2+\sum_{i=-1}^{8}s_i\phi_i
\end{equation*}
We denote a basis for the Milnor ring
\begin{equation*}
\phi_{-1}=xyz;\phi_0=1;\phi_1=x;\phi_2=y;\phi_3=x^2;\phi_4=xy;\phi_5=z;\phi_6=x^2y;\phi_7=xz;\phi_8=yz.
\end{equation*}
The Picard-Fuchs equation for the primitive form is again as in the $P_8$ case. As before, we can compute the first-order expansion of the intergrand as:
\begin{align*}
\Phi_1&=A_1+o(s_{-1},A), &\Phi_2&=A_2+o(s_{-1},A), & \Phi_3&=-s_{-1} A_2/3+A_3+o(s_{-1},A),\\
\Phi_4&=A_4+o(s_{-1},A), &\Phi_5&=-s_{-1} A_4/6+A_5+o(s_{-1},A), &\Phi_6&=A_6+o(s_{-1},A),\\
\Phi_7&=A_7+o(s_{-1},A), &\Phi_8&=A_8+o(s_{-1},A). &&
\end{align*}
By choosing the middle homology cycle properly, we compare the two expressions of the period matrix $\Pi_{\delta,i}$,
and then obtain the second-order expansion:
\begin{align*}
t_{i}&= s_i+o(s^2), i\neq0,1,2,4.&&\\
t_{0}&= s_{0}-s_{3}s_{7}/3-s_{5}^2/6+o(s^2),&
t_{1}&= s_{1}-s_{5}s_{7}/6+o(s^2),\\
t_{2}&= s_{2}-s_{3}s_{-1}/3-s_{6}s_{7}/3-s_{5}s_{8}/3+o(s^2),&
t_{4}&= s_{4}-s_{5}s_{-1}/6-s_{7}s_{8}/6+o(s^2).
\end{align*}
It is easy to check that the genus-0 3-point correlators match those in the A-model. We compute the key basic correlators:
\begin{eqnarray*}
\begin{split}
&\LD x,y,z,xyz\RD_{0,4}=C_{1,2,3,-1}(0)=0,\\
&\LD y,y,y,xyz\RD_{0,4}=C_{2,2,2,-1}(0)=\frac{\partial(y^3)}{\partial s_{-1}}|_{\textbf{t}=0}=-\frac{1}{3}.\\
&\LD x,x,xz,xyz\RD_{0,4}=C_{1,1,7,-1}(0)=\frac{\partial(x^3z)}{\partial s_{-1}}|_{\textbf{t}=0}=-\frac{1}{3}.
\end{split}
\end{eqnarray*}
We define $\bar\Psi_{J_{10}}:\phi_i\mapsto(-1)^{1-d_i}\Psi_{J_{10}}(\phi_i)$, and rescale the primitive form by $-1$. Now,
\begin{equation*}
\bar\Psi_{J_{10}}\Big(F_{0,3}^B\Big)=F_{0,3}^A,
\bar\Psi_{J_{10}}\Big(F_{0,4}^B\Big)=F_{0,4}^A.
\end{equation*}
The LG-to-LG mirror symmetry for $(J_{10}^T,G_{\textrm{max}})$ follows from Theorem \ref{thm:FJRW-reconstruction}.

\subsection{Convergence of FJRW-theory}
For the elliptic singularity $(W,G_{\textrm{max}})$, recall that the \emph{FJRW ancestor correlator function}:
\begin{equation*}
\LD\LD\tau_{l_1}(\alpha_1),\cdots,\tau_{l_n}(\alpha_n)\RD\RD_{g,n}^{W,G_{\textrm{max}}}(\textbf{t})
=\sum_{k=0}\frac{1}{k!}\LD\tau_{l_1}(\alpha_1),\cdots,\tau_{l_n}(\alpha_n),\textbf{t},\cdots,\textbf{t}\RD_{g,n+k}^{W,G_{\textrm{max}}}
\end{equation*}
where $\textbf{t}\in H^{FJRW}_{W,G_{\textrm{max}}}$ (include the complex degree one case). In this subsection, we prove the convergence of the functions near $\textbf{t}=0$.
We first define the \emph{length} for a FJRW genus-0 primary correlator.
\begin{df}
We say a genus-0 n-points FJRW correlator has \emph{length $m$} if it can be reconstructed by genus-0 FJRW correlators with fewer marked points by at most $m+1$ WDVV equations.
\end{df}
Recall the insertions $\alpha_i$ belong to the basis we fixed in subsection 3.1.
We denote
\begin{eqnarray}
\begin{split}
&I_{g,n}:=\max\Big|\LD\alpha_1,\cdots,\alpha_n\RD_{g,n}^{W,G_{\textrm{max}}}\Big|\\
&I_{0,n}(m):=\max\Big|\LD\alpha_1,\cdots,\alpha_n\RD_{0,n}^{W,G_{\textrm{max}}}\big|\LD\alpha_1,\cdots,\alpha_n\RD_{0,n}^{W,G_{\textrm{max}}} \text{ is of length }m\Big|\\
&I_{g,K=n+k,L}:=\max\Big|\int_{\overline{\mathscr{M}}_{g,n+k}}\pi_{n,k}^{*}(\Psi_{g,n,L})\cdot\Lambda_{g,n+k}^{W}(\alpha_1,\cdots,\alpha_{n},\be_{J^{-1}},\cdots,\be_{J^{-1}})\Big|
\end{split}
\end{eqnarray}
where $\Psi_{g,n,L}$ is a monomial of $\psi$ and $\kappa$ classes in $H^*(\overline{\mathscr{M}}_{g,n})$ with $\deg\Psi_{g,n,L}=L$. The Selection rule (\ref{deg:FJRW}) implies nonzero integrals, $L$ is bounded by $g$ and $n$,
$$2g-2\leq\,L\leq 2g-2+n.$$
On the other hand, for $m\geq0$, the WDVV equation (\ref{eq:WDVV-FJRW}) implies
\begin{equation}\label{eq:length-FJRW}
I_{0,n}(m)\leq12I_{0,n}(m-1)+2|I|
\end{equation}
where we have the convention $I_{0,n}(-1)=0$.
\begin{lm}
For $K\leq4$, $I_{0,K}\leq\,C_0$ for some $C_0$. For $K\geq5$, there exists sufficient large constant $C$, such that
\begin{equation}\label{eq:FJRW g=0}
I_{0,K}\leq C^{K-4}(K-5)!
\end{equation}
\end{lm}
\begin{proof}
For $n$ fixed, by degree argument, the length of $\LD\alpha_1,\cdots,\alpha_n\RD_{0,n}^{W,G_{\textrm{max}}}$ is bounded by some $M$, which depends only on $n$. Thus (\ref{eq:length-FJRW}) implies
$$
I_{0,n}\leq\,I_{0,n}(M)<12^{M+1}|I|
$$
we only need to estimate the values of those correlators with fewer insertions.
We use induction on the number of insertions. For $K\leq5$, the estimation holds as there are just finite correlators. Assume the estimation (\ref{eq:FJRW g=0}) holds for all $K\leq k+2$, $k\geq 4$, then the induction is true by the following estimation,
\begin{eqnarray*}
\begin{split}
|I|&\leq 12\sum_{i=2}^{k-2}{k\choose i}I_{0,i+3}I_{0,k-i+3}+24k\,I_{0,4}I_{0,k+2}\\
&\leq\Big(12\sum_{i=2}^{k-2}\frac{k(k-1)}{i(i-1)(k-i)(k-i-1)}+\frac{24kC_0}{k-2}\Big)C^{k-2}(k-2)!\\
&\leq\frac{54+48C_0}{C}C^{k-1}(k-2)!
\end{split}
\end{eqnarray*}
Here we use the inequality
\begin{equation*}
\sum_{i=2}^{k-2}\frac{k(k-1)}{i(i-1)(k-i)(k-i-1)}=\frac{k}{k-1}\sum_{i=2}^{k-2}(\frac{1}{i-1}+\frac{1}{k-i})(\frac{1}{i}+\frac{1}{k-i-1})
\leq\frac{9}{2}
\end{equation*}
\end{proof}

\begin{lm}
For genus-1 primary correlators, we have
\begin{equation*}
I_{1,K}\leq C^{K}K!
\end{equation*}
\end{lm}
\begin{proof}
We only give the proof of $P_8$ case, i.e. for $I_{1,K}^{P_8}$. The other two cases are similar.
It is easy to see the estimation holds for $K=1$. Thus we can use the method of induction, assume it holds for $K\leq k+1$, $k\geq0$.
In this section, we simplify the notation by $\Lambda:=\Lambda_{1,k+4}^{P_8}(\be_{x},\be_{yz},\be_{y},\be_{xz},\be_{xyz},\cdots,\be_{xyz})$. Recall
\begin{equation*}
\LD\be_{xyz},\cdots,\be_{xyz}\RD_{1,k+2}^{P_8}=\frac{1}{3}\int_{\pi_{4,k}^{-1}(\delta_{2,2})}\Lambda
\end{equation*}
The integration of $\Lambda$ on $\pi_{4,k}^{-1}(\delta_{2,3})$ and $\pi_{4,k}^{-1}(\delta_{2,4})$ are both zero. And we also have
\begin{equation*}
\Big|\int_{\pi_{4,k}^{-1}(\delta_{3,4})}
\Lambda\Big|
\leq4\sum_{i=0}^k\ {k\choose i}\,I_{1,i+1}I_{0,3}I_{0,k-i+4}
\leq\frac{4C_0+4C_0^2}{C}C^{k+2}(k+2)!
\leq C^{k+2}(k+2)!
\end{equation*}
Next, we consider the genus-0 contribution. Similarly, we have
\begin{equation*}
\Big|\int_{\pi_{4,k}^{-1}(\delta_{0,3}+\delta_{0,4}-2\delta_{\beta})}\Lambda\Big|
\leq C^{k+2}(k+2)!
\end{equation*}
Now we integrate $\Lambda$ on
$\pi_{4,k}^{-1}(12\delta_{2,2}+4\delta_{2,3}-2\delta_{2,4}+6\delta_{3,4}+\delta_{0,3}+\delta_{0,4}-2\delta_{\beta})$.
Combine all the inequalities above, Getzler's relation implies
\begin{equation*}
I_{1,k+2}^{P_8}
=\max\Big|\LD\be_{xyz},\cdots,\be_{xyz}\RD_{1,k+2}^{P_8}\Big|
\leq C^{k+2}(k+2)!
\end{equation*}
\end{proof}
\begin{lm}
Using $g$-reduction, for $K=n+k$, $L=\sum_{i=1}^{n}l_i$, we have
\begin{equation*}
\Big|\LD\tau_{l_1}(\alpha_1),\cdots,\tau_{l_n}(\alpha_n),\be_{J^{-1}},\cdots,\be_{J^{-1}}\RD_{g,K}^{W}\Big|
\leq I_{g,K,L}\leq C^{2g-2+K+L}(2g-2+K+L)!
\end{equation*}
Here the constant $C$ is sufficiently large, and only depends increasingly on $g,n$.
\end{lm}
\begin{proof}
For fixed $g$ and $n$, we can use induction on $L$. In case of $L=0$, non-vanishing correlators must be of genus-0 or genus-1, thus the estimation follows from two lemmas above.
For $L\geq1$, according to $g$-reduction, $\Psi_{l_1,\cdots,l_n}:=\prod_{i=1}^n\psi_i^{l_i}$ can be represented by a linear combination of dual graphs, each of which has at least one edge. The number of dual graphs is only depends on $g$ and $n$, but not $k$. We denote by $N_{g,n}$.
We can induct on the number of edges in the dual graph. For one edge case, if there are two components, $C_i$, $i=1,2$. component $C_i$ has genus $g_i$, $n_i$ number of insertions from the $n$ insertion, and the degree of $\psi,\kappa$ classes is $L_i$. Those are all fixed by the dual graph. However $k$ copies of insertion $\be_{J^{-1}}$ can be distributed to either $C_1$ or $C_2$. We denote the number of copies in $C_i$ by $k_i$.  Thus we have
$$g=g_1+g_2, n=n_1+n_2, k=k_1+k_2, L_1+L_2=L-1
$$
Now, $L_i<L$. By induction, we have the total bound
\begin{eqnarray*}
\begin{split}
&N_{g,n}\sum_{k_1=0}^k{k\choose k_1}\,I_{g_1,n_1+k_1+1,L_1}\,I_{g_2,n_2+k_2+1,L_2}\\
&\leq\,N_{g,n}C^{2g-3+n+k+L}\sum_{k_1=0}^k{k\choose\,k_1}\,(2g_1-1+n_1+k_1+L_1)!\,(2g_2-1+n_2+k_2+L_2)!\\
&=N_{g,n}C^{2g-3+n+k+L}\frac{(2g_1-1+n_1+L_1)!(2g_2-1+n_2+L_2)!}{(2g-2+n+L)!}(2g-2+n+k+L)!\\
&\leq\,C^{2g-2+n+k+L}(2g-2+n+k+L)!
\end{split}
\end{eqnarray*}
The second equality is using the Chu-Vandemonde equality:
\begin{equation}
_2 F_1(-k,b;c;1)=\frac{(c-b)_k}{(c)_k}=\frac{(c-b)(c-b+1)\cdots(c-b+k-1)}{(c)(c+1)\cdots(c+k-1)}
\end{equation}
where we set $b=2g_1+n_1+L_1, c=-k-2g_2-n_2-L_2+1$.
If there is just one component, then the total bound is
\begin{equation*}
N_{g,n}I_{g-1,K+2,L-1}\leq\,N_{g,n}C^{2g-2+K+L-1}(2g-2+K+L-1)!\leq C^{2g-2+K+L}(2g-2+K+L)!
\end{equation*}
\end{proof}
\begin{lm}
The FJRW ancestor correlator $\LD\LD\tau_{l_1}(\alpha_1),\cdots,\tau_{l_n}(\alpha_n)\RD\RD_{g,n}^{W,G}(s\be_{J^{-1}})$ is convergent at $s=0$.
\end{lm}
\begin{proof}By definition and the previous lemma, we have
\begin{eqnarray*}
\begin{split}
&\Big|\LD\LD\tau_{l_1}(\alpha_1),\cdots,\tau_{l_n}(\alpha_n)\RD\RD_{g,n}^{W,G}(s\be_{J^{-1}})\Big|\\
&=\Big|\sum_{k=0}\frac{1}{k!}\LD\tau_{l_1}(\alpha_1),\cdots,\tau_{l_n}(\alpha_n),s\be_{J^{-1}},\cdots,s\be_{J^{-1}}\RD_{g,n+k}^{W,G}\Big|\\
&\leq\sum_{k=0}^{\infty}C^{2g-2+n+k+L}\frac{(2g-2+n+k+L)!}{k!}s^k\\
\end{split}
\end{eqnarray*}
Thus the ancestor function is convergent in $|s|<\frac{1}{2C}$.
\end{proof}
Now we prove the convergence of the FJRW part in Theorem \ref{thm:convergent}.
\begin{proof}
We can assume $\textbf{t}=s\be_{J^{-1}}+\sum_{i\geq0}t_i\phi_i$, where $\phi_i$ ranges over all elements in the basis except $\be_{J^{-1}}$. By Selection rule (\ref{deg:FJRW}), $\LD\LD\tau_{l_1}(\alpha_1),\cdots,\tau_{l_n}(\alpha_n)\RD\RD_{g,n}^{W,G}(\textbf{t})$ can be written as a polynomial of $t_i$, with each coefficient some FJRW ancestor function valued at $s\be_{J^{-1}}$. The degree of each monomial just depends on $g$ and $n$. Now the convergence is an easy consequence from the previous lemma.
\end{proof}

\section{LG-to-CY mirror theorem of all genera}


For more details about Gromov-Witten theory of orbifolds, we refer to \cite{AbGV}, \cite{ALR} and \cite{ChenR}.
\subsection{Chen-Ruan cohomology ring for elliptic orbifold $\mathbb{P}^1$}

We consider the elliptic singularities $W:=P_8,X_9,J_{10}$ from the previous section,
\begin{eqnarray}
\begin{split}
&P^T_8=P_8: x^3+y^3+z^3.\\
&X_{9}^T: x^2y+y^3+xz^2. \\
&J_{10}^T: x^3+y^3+xz^2.
\end{split}
\end{eqnarray}
The hypersurface $X_W$ in weighted projective space is defined by:
\begin{equation*}
X_W=\{W=0\}\subseteq \mathbb{P}^2
\end{equation*}
The first Chern class $c_1(TX_W)=0$. Thus $X_W/\widetilde{G}_{\textrm{max}}$ is an elliptic orbifold $\mathbb{P}^1$. We denote those orbifolds by $\mathcal{X}:= \mathbb{P}^1_{3,3,3}, \mathbb{P}^1_{4,4,2}, \mathbb{P}^1_{6,3,2}$. More precisely, we have:
\begin{eqnarray*}
&\mathbb{P}^1_{3,3,3}&=\{x^3+y^3+z^3=0\}/\widetilde{G}_{\textrm{max}}\\
&\mathbb{P}^1_{4,4,2}&=\{x^2y+y^3+xz^2=0\}/\widetilde{G}_{\textrm{max}}\\
&\mathbb{P}^1_{6,3,2}&=\{x^3+y^3+xz^2=0\}/\widetilde{G}_{\textrm{max}}.
\end{eqnarray*}
It is easy to see that the $j$-invariant of the three hypersurface is 0,1728 and 0 respectively. On the other hand, $\widetilde{G}_{\textrm{max}}$ are $\mathbb{Z}/3\mathbb{Z}\times\mathbb{Z}/3\mathbb{Z},\mathbb{Z}/4\mathbb{Z}$ and $\mathbb{Z}/6\mathbb{Z}$ in each case.
Next, we consider the Chen-Ruan cohomology for the three orbifolds. We recall that the \emph{Chen-Ruan cohomology} of $X_W/\widetilde{G}_{\textrm{max}}$, denoted as $H^*_{CR}(X_W/\widetilde{G}_{\textrm{max}})$, is defined to be the cohomology of its \emph{inertial orbifold} $\wedge (X_W/\widetilde{G}_{\textrm{max}})$: $H^*(\wedge (X_W/\widetilde{G}_{\textrm{max}}))$.

In those cases, the Chen-Ruan cohomology has a canonical set of basis $\mathscr{B}$:
\begin{equation*}
1\in H^0_{CR}(\mathcal{X}),\Delta_{x,1},\cdots,\Delta_{x,p},\Delta_{y,1},\cdots,\Delta_{y,q},\Delta_{z,1},\cdots,\Delta_{z,r}, \cal{P}.
\end{equation*}
where $p,q,r$ are the orders of the three orbifold points, and $\cal{P}$ is the generator of $H^2_{CR}(\mathcal{X})$.
The complex degree of $1$ is zero and the complex degree of $\cal{P}$ is one. The degrees of other cases are shifted. Their complex degrees are
$$\deg\Delta_{x,i}=\frac{i}{p},\deg\Delta_{y,i}=\frac{i}{q}, \deg\Delta_{z,i}=\frac{i}{r}.$$

\subsubsection{Dimension, String, Divisor and WDVV}

$\overline{\mathscr{M}}_{g,n}(\mathcal{X},\beta)$ is the moduli space of stable maps from genus-$g$, $n$-points orbi-curve to $\mathcal{X}$ with the push-forward of the fundamental class is $\beta$, which lies in the Mori cone of the homological classes of effective 1-cycles. In our case, as $\mathcal{X}$ is 1-dimensional, we can write $\beta=d\cal{P}$, for $d\geq0$. We sometimes use $d$ to replace $\beta$. The virtual dimension of $\overline{\mathscr{M}}_{g,n}(\mathcal{X},\beta)$ is
\begin{equation}\label{form:GW dim}
\textrm{vir}\dim_{\mathbb{C}}\overline{\mathscr{M}}_{g,n}(\mathcal{X},\beta)=(3-1)(g-1)+n+c_1(T\mathcal{X})\cdot\beta=2g-2+n
\end{equation}

The \emph{Gromov-Witten ancestor correlators} are defined as integrals over $\overline{\mathscr{M}}_{g,n}(\mathcal{X},\beta)$:
\begin{equation}
\LD\tau_{l_1}(\alpha_1),\cdots,\tau_{l_n}(\alpha_n)\RD_{g,n,\beta}^{\mathscr{X}}
=\int_{\big[\overline{\mathscr{M}}_{g,n}(\mathcal{X},\beta)\big]^{\mathrm{vir}}}\prod_{i=1}^n\bar{\psi_{i}}^{l_i} ev^*_i(\alpha_i)
\end{equation}
Here $\alpha_i\in H^*_{CR}(\mathcal{X})$. By multi-linearity, we can always choose $\alpha_i$ from the canonical set $\mathscr{B}$. $\bar\psi_{i}:=\pi^*(\psi_i)$, with
$\pi: \overline{\mathscr{M}}_{g,n}(\mathcal{X},\beta)\rightarrow\overline{\mathscr{M}}_{g,n}.$
And $\textrm{ev}_i:\overline{\mathscr{M}}_{g,n}(\mathcal{X},\beta)\longrightarrow\wedge\mathcal{X}$ is evaluation map to the $i$-th marked point.

With formal variables $t_{i,l_i}, q$ and $\hbar$, the \emph{genus-$g$ Gromov-Witten potential} is
\begin{equation}
\cal{F}_g^{GW}(\mathcal{X})
=\sum_{n}\sum_{l_i}\sum_{\alpha_i}\sum_{\beta}\LD\tau_{l_1}(\alpha_1),\cdots,\tau_{l_n}(\alpha_n)\RD_{g,n,\beta}\frac{\prod_{i=1}^n t_{i,l_i}}{\prod_{i=1}^n l_i !}q^{\beta}
\end{equation}
The \emph{total Gromov-Witten ancestor potential} is
\begin{equation}
\cal{A}_{GW}(\mathcal{X})=\exp\Big(\sum_{g\geq0}\hbar^{2g-2}\cal{F}_g^{GW}(\mathcal{X})\Big)
\end{equation}
The \emph{genus-$g$ ancestor correlator function} is
\begin{equation*}
\LD\LD\tau_{l_1}(\alpha_1),\cdots,\tau_{l_n}(\alpha_n)\RD\RD_{g,n}(\textbf{t}):
=\sum_{k\geq0,\beta}\frac{q^{\beta}}{k!}
\int_{\big[\overline{\mathscr{M}}_{g,n+k}(\mathcal{X},\beta)\big]^{\mathrm{vir}}}\prod_{i=1}^n\bar\psi_{i}^{l_i} ev^*_i(\alpha_i)\prod_{i=n+1}^{n+k}ev^*_i(\textbf{t})
\end{equation*}
Ancestor correlators satisfy WDVV, string, dilaton and divisor equations as well. For example, the divisor equation is
\begin{equation}\label{eq:divisor}
\LD\tau_{l_1}(\gamma_1)\,\cdots\,\tau_{l_n}(\gamma_n),\cal{P}\RD_{g,n+1,d}=d\,\LD\tau_{l_1}(\gamma_1)\,\cdots\,\tau_{l_n}(\gamma_n)\RD_{g,n,d}
\end{equation}

The \emph{genus-$g$ ancestor potential function} is
\begin{equation*}
\cal{F}_g^{GW}(\textbf{t}):=\sum_{n}\sum_{\alpha_1,\cdots,\alpha_n}\frac{1}{n!}\LD\LD\tau_{l_1}(\alpha_1),\cdots,\tau_{l_n}(\alpha_n)\RD\RD_{g,n}(\textbf{t})
\end{equation*}
The \emph{total ancestor potential function} is
\begin{equation*}
\cal{A}_{GW}(\hbar;\textbf{t}):=\exp\sum_{g=0}^{\infty}\hbar^{2g-2}\cal{F}_g^{GW}(\textbf{t})
\end{equation*}

\subsubsection{Chen-Ruan product}

For $\alpha=(p,g)\in H_{CR}^*(\mathscr{X})$, we denote $\textrm{Supp}(\alpha)=p$, $\alpha'=(p,g^{-1})$ and $\alpha^{i}=(p,g^i)$. We also denote by $|\alpha|$ the order of isotropy group for the stacky point fixed by $\alpha$.
For $\alpha,\beta\in H_{CR}^*(\mathscr{X})$, the paring is
$$\LD\alpha,\beta\RD=\frac{1}{|\alpha|}\delta_{\alpha,\beta'}.
$$
The Chen-Ruan product is defined by the relation
$$\LD\alpha\star \beta, \gamma\RD
=\LD\alpha, \beta, \gamma\RD_{0,3,0}.$$
The Chen-Ruan product can be easily calculated. Except for those from the pairing, they are
\begin{equation}\label{eq:Chen-Ruan}
\LD\alpha,\beta,\gamma\RD_{0,3,0}=\left\{
\begin{array}{ll}
|\alpha|^{-1},&\alpha,\beta,\gamma \text{ have the same support and total degree is 1};\\
0,&\text{other cases}.
\end{array}
\right.
\end{equation}
Thus $\alpha\star\beta\neq0$ if and only if $\textrm{Supp}(\alpha)=\textrm{Supp}(\beta)$ and $\deg\alpha+\deg\beta\leq1$.
Hence for primitive $\alpha$, $\alpha^i$ is the Chen-Ruan product of $i$-copies of $\alpha$, $\alpha^i=\alpha\star\cdots\star\alpha$.

\subsection{Reconstruction}

Our method of reconstructing Gromov-Witten theory is similar to that of FJRW-theory described in the last section. However, Gromov-Witten
theory has an additional parameter-degree. Every step of the reconstruction becomes more delicate. Since we need it for the proof of convergence,
we reproduce the details here.

The general idea is the similar as that of FJRW-theory.
We use WDVV, string and divisor equation (which does not exists in FJRW-theory), to reconstruct genus-0 primary correlators from genus-0 $n$-point basic correlators, with $n\leq 4$.
We define
\begin{df}
We call a class $\gamma$ \emph{primitive} if it cannot be written as $\gamma=\gamma_1\star\gamma_2$ for $0 < \deg\gamma_1 < \deg\gamma.$
\end{df}
\begin{df}
We call a correlator \emph{basic} if there are no insertions of $1,\cal{P}$ and at most two non-primitive insertions.
\end{df}
\begin{df}
We call a genus-0 primary correlator is \emph{reconstructable} if it can be expressed by linear combinations of products of Chen-Ruan product structural constants and $\LD\Delta_{x,1},\Delta_{y,1},\Delta_{z,1}\RD_{0,3,1}^{\mathscr{X}}$, only using WDVV, string and divisor equation.
\end{df}

We first recall the WDVV equation for elliptic orbifold $\mathbb{P}^1$. Set $S=\{1,\cdots,n\}$, $n\geq1$, for $d\geq1$, we have:
\begin{equation}\label{eq:WDVV-GW}
\begin{split}
\LD\gamma_1,\gamma_2,&\delta_S,\gamma_3\star\gamma_4\RD_{0,n+3,d}=I_0+I_1+I_2+I_3
\end{split}
\end{equation}
where
\begin{eqnarray*}
\begin{split}
I_0=&\LD\gamma_1,\gamma_3,\delta_{S},\gamma_2\star\gamma_4\RD_{0,n+3,d}+\LD\gamma_1\star\gamma_3,\delta_{S},\gamma_2,\gamma_4\RD_{0,n+3,d}-\LD\gamma_1\star\gamma_2,\delta_{S},\gamma_3,\gamma_4\RD_{0,n+3,d}\\
I_1=&\sum_{\gamma_2\rightleftarrows\gamma_3}\textrm{Sign}(\gamma_2,\gamma_3)\sum_{\begin{subarray}{l}A\sqcup\,B=S(n)\\A,B\neq\emptyset,i=0,d\end{subarray}}\Big(\LD\gamma_1,\gamma_3,\delta_A,\mu\RD_{0,|A|+3,d-i}\eta^{\mu,\nu}\LD\nu,\delta_B,\gamma_2,\gamma_4\RD_{0,n+3-|A|,i}\Big)\\
I_2=&\sum_{\gamma_2\rightleftarrows\gamma_3}\textrm{Sign}(\gamma_2,\gamma_3)\sum_{\begin{subarray}{l}A\sqcup\,B=S(n)\\0<i<d\end{subarray}}\LD\gamma_1,\gamma_3,\delta_A,\mu\RD_{0,|A|+3,d-i}\eta^{\mu,\nu}\LD\nu,\delta_B,\gamma_2,\gamma_4\RD_{0,n+3-|A|,i}\\
I_3=&\sum_{\gamma_2\rightleftarrows\gamma_3}\textrm{Sign}(\gamma_2,\gamma_3)
\Big(\LD\gamma_1,\gamma_3,\mu\RD_{0,3,d}\eta^{\mu,\nu}\LD\nu,\delta_{S},\gamma_2,\gamma_4\RD_{0,n+3,0}+\LD\gamma_1,\gamma_3,\delta_{S},\mu\RD_{0,n+3,0}\eta^{\mu,\nu}\\
&\cdot\LD\nu,\gamma_2,\gamma_4\RD_{0,3,d}\Big)
\end{split}
\end{eqnarray*}
Note that for $d=0$, the WDVV equation is modified to be
\begin{eqnarray}\label{eq:WDVV-GW d=0}
\begin{split}
&\LD\gamma_1,\gamma_2,\delta_S,\gamma_3\star\gamma_4\RD_{0,n+3,0}\\
=&\LD\gamma_1,\gamma_3,\delta_{S},\gamma_2\star\gamma_4\RD_{0,n+3,0}+\LD\gamma_1\star\gamma_3,\delta_{S},\gamma_2,\gamma_4\RD_{0,n+3,0}-\LD\gamma_1\star\gamma_2,\delta_{S},\gamma_3,\gamma_4\RD_{0,n+3,0}
\end{split}
\end{eqnarray}
Once we have more than two non-primitive insertions, we can choose $\gamma_1,\gamma_2, \gamma_3\star\gamma_4$ to be these, where $\gamma_4=\gamma_3^i$ for some $1\leq\,i\leq|\gamma_3|-2$.
If there are other nonprimitive insertions with different fixed points, we can choose $\gamma_1,\gamma_2$ to be these insertions. Otherwise, we choose the smallest degree nonprimitive insertion to be $\gamma_3^{i+1}$.
\begin{enumerate}
\item For $\mathbb{P}^1_{3,3,3}$, each primitive insertion has degree $1/3$, each term in $I_0$ either vanishes or has an insertion $\cal{P}$.
\item For $\mathbb{P}^1_{4,4,2}$, each primitive insertion has degree $1/4$.
\item For $\mathbb{P}^1_{6,3,2}$, each primitive insertion has degree $1/3$, or $1/6$. For the $1/3$ case, it is the same as in $\mathbb{P}^1_{3,3,3}$.
\end{enumerate}

\subsubsection{Recursion for genus-0 3-point and 4-point basic correlators}
Now we classify the genus-0 3-point or 4-point basic correlators into six types. Here $\alpha,\beta,\gamma,\xi$ are all primitive elements:
\begin{enumerate}
\item [Type 1]: $\LD\alpha,\beta^{j},\gamma,\xi^{i+1}\RD_{0,4,d}$, $i,j\geq1$, supports are not the same point.
\item [Type 2]: $\LD\Delta_{x,1},\Delta_{y,1},\Delta_{z,1}\RD_{0,3,d}$.
\item [Type 3]: $\LD\gamma,\gamma,\gamma',\gamma'\RD_{0,4,d}$, $|\gamma|$ is greatest among all primitive elements.
\item [Type 4]: $\LD\alpha,\beta^i,\beta^j\RD_{0,3,d}$, $\alpha\neq\beta$.
\item [Type 5]: $\LD\beta,\beta,\beta',\beta'\RD_{0,4,d}$, $|\beta|=3$ in case of $\mathbb{P}^1_{6,3,2}$ or $|\beta|=2$.
\item [Type 6]: $\LD\alpha,\alpha^i,\alpha^j\RD_{0,3,d}$.
\end{enumerate}

\begin{lm}For all $\mathscr{X}=\mathbb{P}^1_{3,3,3},\mathbb{P}^1_{4,4,2},\mathbb{P}^1_{6,3,2},$ we have
\begin{equation}\label{eq:n=3,d=1}
\LD\Delta_{x,1},\Delta_{y,1},\Delta_{z,1}\RD_{0,3,1}^{\mathscr{X}}=1.
\end{equation}
\end{lm}
\begin{proof}
For $\langle\Delta_{x,1},\Delta_{y,1},\Delta_{z,1}\rangle^{\mathbb{P}^1_{3,3,3}}_{0,3,1}$, as in \cite{ALR}, we consider the R-equivalence class of principal $\mathbb{Z}/3\mathbb{Z}$-bundles over orbifold $\mathbb{P}^1_{3,3,3}$ with a $\mathbb{Z}/3\mathbb{Z}$-equivariant map to a genus-1 curve. There is just one such equivalent class, thus $\langle\Delta_{x,1},\Delta_{y,1},\Delta_{z,1}\rangle^{\mathbb{P}^1_{3,3,3}}_{0,3,1}=1.$
The other two cases are obtained similarly.
\end{proof}
Now we start to reconstruct the genus-0 4-point correlators with degree $0$,
\begin{lm}\label{lm:vanish n=4}
For Type 1 correlators, WDVV equation implies
\begin{equation}\label{eq:vanish n=4}
\LD\alpha,\beta^{j},\gamma,\xi^{i+1}\RD_{0,4,0}=0.
\end{equation}
Type 3 correlators $\LD\gamma,\gamma,\gamma',\gamma'\RD_{0,4,0}$ can be reconstructed from Chen-Ruan product and $\LD\Delta_{x,1},\Delta_{y,1},\Delta_{z,1}\RD_{0,3,1}$,
\begin{equation}\label{eq:nonvanish n=4}
\LD\gamma,\gamma,\gamma',\gamma'\RD_{0,4,0}=-|\gamma|^{-2}.
\end{equation}
\end{lm}
\begin{proof}
For Type 1 correlator, if there are three primitive insertions, i.e. $j=1$, then it is either
$\LD\Delta_{z,1},\Delta_{y,1},\Delta_{y,1},\Delta_{x,4}\star\Delta_{x,1}\RD_{0,4,0}^{\mathbb{P}^1_{2,3,6}}$, or
$\LD\Delta_{z,1},\Delta_{z,1},-,\xi^i\star\xi\RD_{0,4,0}^{\mathscr{X}}$.
Then applying WDVV equation (\ref{eq:WDVV-GW d=0}), they will vanish.

For other cases, i.e. $j\geq2$, we can assume $j>i$ if $\beta=\xi$. According to dimension axiom, we will always have
\begin{equation}
\deg\alpha\geq\deg\gamma, \alpha\neq\xi
\end{equation}
We apply WDVV equation (\ref{eq:WDVV-GW d=0}) for $\LD\alpha,\beta^j,\gamma,\xi^i\star\xi\RD_{0,4,0}$. On the right hand side of the equation,
the second term vanishes because $\alpha\neq\xi$.
The first term is
$\LD\alpha\star\beta^{j},\gamma,\xi^i,\xi\RD_{0,4,0}$, it also vanishes. Or else we must have
$\alpha=\beta$ and $(j+1)\deg\alpha<1$. However, $\deg\alpha\geq\deg\gamma$, which implies $\deg\alpha+\deg(\beta^j)+\deg\gamma\leq1$. This contradicts with dimension axiom for nonvanishing correlators.
The last term will either vanish or equal to $\LD\alpha,\xi^{j+1},\gamma,\xi^{i}\RD_{0,4,0}$, we can continue to apply (\ref{eq:WDVV-GW d=0}) again and again, unless the second insertion is $\cal{P}$ or the last insertion is primitive, both correlators are zero.

For Type 3 correlator $\LD\gamma,\gamma,\gamma',\gamma'\RD_{0,4,0}$, let $\alpha,\beta$ be the other two primitive insertions and we apply WDVV equation (\ref{eq:WDVV-GW}) to $\LD\gamma,\gamma,\gamma',\gamma'\RD_{0,4,d}$ for $d=1$. The equation (\ref{eq:nonvanish n=4}) follows from divisor axiom, equation (\ref{eq:n=3,d=1}) and (\ref{eq:vanish n=4}).



\end{proof}

\begin{lm}\label{lm:key-GW}
All basic correlators are reconstructable for $d\geq1$.
\end{lm}
For Type 1 correlators, go through the proof of Lemma \ref{lm:vanish n=4}, take any degree $d>0$, the reconstruction follows.

For Type 6, it is the same if we can reconstruct $\LD\alpha,\alpha^i,\alpha^j,\cal{P}\RD_{0,3,d}$. Like we did for the basic correlators, we reduce to the case $\LD\cal{P},\alpha,\alpha,\alpha^j,\cal{P}\RD_{0,3,d}$. We can choose $\beta$ such that $\deg\alpha^i+\deg\beta'\geq1$. Thus $I_0,I_1,I_3$ all vanish,
\begin{eqnarray}\label{eq:Type6}
\begin{split}
&\frac{d^{2}}{|\beta|}\LD\alpha^{j},\alpha,\alpha\RD_{0,3,d}=\langle\cal{P},\alpha^{j},\alpha,\alpha,\beta\star\beta'\rangle_{0,5,d}\\
&=2\sum_{i=1}^{d-1}(d-i)\Big(\langle\beta,\alpha,\mu\rangle_{0,3,d-i}\eta^{\mu,\nu}\langle\nu,\alpha,\beta',\alpha^j\rangle_{0,4,i}
-\langle\alpha^{j},\alpha,\mu\rangle_{0,3,d-i}\eta^{\mu,\nu}\langle\nu,\alpha,\beta,\beta'\rangle_{0,4,i}\Big)
\end{split}
\end{eqnarray}
The choices of the insertions for each type are slightly different. For the reader's convenience, we put the proof in the appendix.

\subsubsection{Genus-0 resconstruction}
We have the following lemma.
\begin{lm}
The WDVV equation and the divisor equation imply that all the genus-0 correlators for $\mathbb{P}^1_{3,3,3}, \mathbb{P}^1_{4,4,2}, \mathbb{P}^1_{6,3,2}$ are uniquely determined by the pairing, the genus-0 3-point and 4-point correlators.
\end{lm}
\begin{proof}
Let us denote by $P$ the maximum complex degree of any primitive class, and by $Q$ is the maximum complex degree of any homogeneous non-divisor class. Similarly as we did for FJRW theory, we can use WDVV, plus the string equation and divisor equation to reconstruct genus-0 primary correlators from the Chen-Ruan product structural constants and basic correlators.

Now let $\LD\gamma_1,\cdots,\gamma_{n}\RD_{0,n,d}^{\mathcal{X}}$ be a basic correlator with first $n-2$ insertions primitive. Thus, $\deg\gamma_i\leq P$ for $i \leq n-2$ and $\deg\gamma_{n-1},\deg\gamma_n\leq Q$. By the dimension counting,
\begin{equation*}
n-2\leq (n-2)P+2Q.
\end{equation*}
It is easy to obtain the data $P$, $Q$ for each orbifold:
\begin{align*}
&\mathbb{P}^1_{3,3,3}:P=\frac{1}{3},Q=\frac{2}{3}.
&\mathbb{P}^1_{4,4,2}:P=\frac{1}{2},Q=\frac{3}{4}.
&\mathbb{P}^1_{6,3,2}:P=\frac{1}{2},Q=\frac{5}{6}.
\end{align*}
Thus we have $n=4$, for $\mathbb{P}^1_{3,3,3}$ and $n=5$ for $\mathbb{P}^1_{4,4,2}$ and $\mathbb{P}^1_{6,3,2}.$
We list all the basic genus-0 five-point correlators. Up to symmetry, the nonvanishing correlators are
$\LD\Delta_{z,1},\Delta_{z,1},\Delta_{y,1},\Delta_{x,5},\Delta_{x,5}\RD_{0,5,d}^{\mathbb{P}^1_{6,3,2}}$,
and
$\LD\Delta_{z,1},\Delta_{z,1},\Delta_{z,1},\alpha,\beta\RD_{0,5,d}^{\mathscr{X}}$,
where
\begin{equation*}
(\alpha,\beta)=\left\{
\begin{array}{ll}
(\Delta_{x,3},\Delta_{x,3}),(\Delta_{x,3},\Delta_{y,3}), & \mathscr{X}=\mathbb{P}^1_{4,4,2};\\
(\Delta_{x,5},\Delta_{x,4}),(\Delta_{x,5},\Delta_{y,2}), & \mathscr{X}=\mathbb{P}^1_{6,3,2}.
\end{array}
\right.
\end{equation*}
It follows by applying WDVV (\ref{eq:WDVV-GW}), that all the correlators above can be reconstructed from genus-0 correlators with less than five insertions by choosing some $\gamma_i, i=1,2,3,4$.
For example, we choose $\gamma_1=\Delta_{x,2}, \gamma_2=\Delta_{x,1}$, $\gamma_3=\Delta_{y,3}$, $\gamma_4=\Delta_{z,1}$ for  $\LD\Delta_{z,1},\Delta_{z,1},\Delta_{z,1},\Delta_{x,3},\Delta_{y,3}\RD_{0,5,d}^{\mathbb{P}^1_{4,4,2}}$.
\end{proof}
Furthermore, we continue the reconstruction on basic correlators and finally have:
\begin{thm}\label{thm:GW recons}
The WDVV-equation and divisor equation imply that all genus-0 primary Gromov-Witten invariants for $\mathscr{X}=\mathbb{P}^1_{3,3,3},\mathbb{P}^1_{4,4,2},\mathbb{P}^1_{6,3,2}$ can be reconstructed from the Chen-Ruan product and $\langle\Delta_{x,1},\Delta_{y,1},\Delta_{z,1}\rangle_{0,3,1}^{\mathscr{X}}=1.$
\end{thm}
\subsubsection{Reconstruction of genus-1 primary potential}

Here we prove the reconstruction theorem for primary genus-1 Gromov-Witten invariants for elliptic orbifold $\mathbb{P}^1$.
Again, we use the Getzler relation.
\begin{thm}\label{thm:g=1 GW}
For $\mathscr{X}=\mathbb{P}^1_{3,3,3},\mathbb{P}^1_{4,4,2},\mathbb{P}^1_{6,3,2}$, the Getzler relation and divisor axiom imply that the genus-1 Gromov-Witten correlators of $\mathscr{X}$ can be reconstructed from genus-0 Gromov-Witten correlators.
\end{thm}
We consider the nonzero genus-1 correlator $\LD\gamma_1,\cdots,\gamma_{n}\RD_{1,n,d}^{\mathcal{X}}$.
As $\mathcal{X}$ is an elliptic orbifold $\mathbb{P}^1$ here, we have $\deg\gamma_{i}\leq 1$. According to the dimension formula (\ref{form:GW dim}), the correlator is nonzero only if every $\gamma_{i}$ is $\cal{P}$. For $d\neq0$, the genus-1 primary correlators are nonzero only if they are of type $\LD\cal{P},\cdots,\cal{P}\RD_{1,n,d}^{\mathcal{X}}$. By the divisor axiom, we have:
\begin{equation*}
\LD\cal{P},\cdots,\cal{P}\RD_{1,n,d}^{\mathcal{X}}=d^{n-1}\LD\cal{P}\RD_{1,1,d}^{\mathcal{X}}.
\end{equation*}
\begin{rem}
$\LD\cal{P}\RD_{1,1,0}^{\mathcal{X}}=-\frac{1}{24}$ and $\LD\cal{P},\cdots,\cal{P}\RD_{1,n,0}^{\mathcal{X}}=0$ for $n>1$.
\end{rem}
Now, we give the proof of Theorem \ref{thm:g=1 GW} by reconstructing $\LD\cal{P}\RD_{1,1,d}^{\mathcal{X}}$.
\begin{proof}
{\bf $\mathbb{P}^1_{3,3,3}$-case:}
We choose four insertions $\Delta_{x,2},\Delta_{x,2},\Delta_{y,1},\Delta_{z,1}\in\,H^*_{CR}(\mathbb{P}^1_{3,3,3})$, and we simply denote by $\Delta_{2,2;1;1}$. We integrate the class $\Lambda_{1,4,d}^{\mathbb{P}^1_{3,3,3}}(\Delta_{2,2;1;1})$ over codimension 2 strata of $\overline{\mathscr{M}}_{1,4}$. For $\delta_{3,4}$, the contribution comes from four decorated dual graphs:
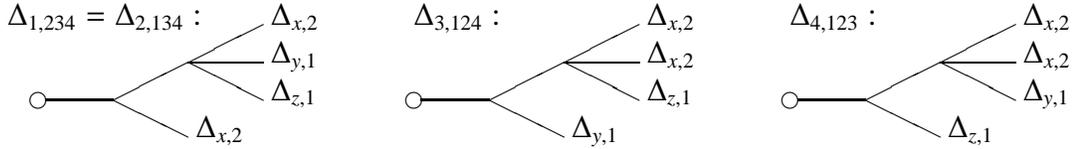
\begin{center}
\begin{picture}(50,23)
    \put(-40,9){\circle{2}}

	\put(-39,9){\line(1,0){9}}
    \put(-30,9){\line(2,1){10}}
    \put(-30,9){\line(2,-1){10}}
    \put(-20,14){\line(2,1){10}}
    \put(-20,14){\line(1,0){10}}
    \put(-20,14){\line(2,-1){10}}


	\put(-9,19){$\Delta_{x,2}$}
	\put(-9,14){$\Delta_{y,1}$}
	\put(-9,9){$\Delta_{z,1}$}
	\put(-19,4){$\Delta_{x,2}$}


	\put(-44,19){$\Delta_{1,234}=\Delta_{2,134}:$}


    \put(10,9){\circle{2}}

	\put(11,9){\line(1,0){9}}
    \put(20,9){\line(2,1){10}}
    \put(20,9){\line(2,-1){10}}
    \put(30,14){\line(2,1){10}}
    \put(30,14){\line(1,0){10}}
    \put(30,14){\line(2,-1){10}}


	\put(41,19){$\Delta_{x,2}$}
	\put(41,14){$\Delta_{x,2}$}
	\put(41,9){$\Delta_{z,1}$}
	\put(31,4){$\Delta_{y,1}$}


	\put(10,19){$\Delta_{3,124}:$}

    \put(60,9){\circle{2}}

	\put(61,9){\line(1,0){9}}
    \put(70,9){\line(2,1){10}}
    \put(70,9){\line(2,-1){10}}
    \put(80,14){\line(2,1){10}}
    \put(80,14){\line(1,0){10}}
    \put(80,14){\line(2,-1){10}}


	\put(91,19){$\Delta_{x,2}$}
	\put(91,14){$\Delta_{x,2}$}
	\put(91,9){$\Delta_{y,1}$}
	\put(81,4){$\Delta_{z,1}$}


	\put(60,19){$\Delta_{4,123}:$}
\end{picture}
\end{center}
Let us fix the total degree is $d+1$. Then
\begin{eqnarray*}
\begin{split}
&\int_{[\Delta_{1,234}]}\Lambda_{1,4,d+1}^{\mathbb{P}^1_{3,3,3}}(\Delta_{2,2;1;1})\\
&=\sum_{d_1+d_2+d_3=d+1}
\LD\cal{P}\RD_{1,1,d_1}^{\mathbb{P}^1_{3,3,3}}\eta^{\cal{P},1}\LD 1,\Delta_{x,2},\Delta_{x,1}\RD_{0,3,d_2}^{\mathbb{P}^1_{3,3,3}}\eta^{\Delta_{x,1},\Delta_{x,2}}\LD\Delta_{2,2;1;1}\RD_{0,4,d_3}^{\mathbb{P}^1_{3,3,3}}\\
&=\sum_{i=0}^{d+1}
\LD\cal{P}\RD_{1,1,i}^{\mathbb{P}^1_{3,3,3}}\LD\Delta_{2,2;1;1}\RD_{0,4,d+1-i}^{\mathbb{P}^1_{3,3,3}}
\end{split}
\end{eqnarray*}
Then we use the result from genus-0 recursion that
$$\LD\Delta_{2,2;1;1}\RD_{0,4,0}^{\mathbb{P}^1_{3,3,3}}=0, \LD\Delta_{2,2;1;1}\RD_{0,4,1}^{\mathbb{P}^1_{3,3,3}}=\frac{1}{3}.$$
Overall, we have
\begin{equation}\label{eq:genus-1}
\int_{\delta_{3,4}}\Lambda_{1,4,d+1}^{\mathbb{P}^1_{3,3,3}}(\Delta_{2,2;1;1})
=\frac{4}{3}\LD\cal{P}\RD_{1,1,d}^{\mathbb{P}^1_{3,3,3}}+4\sum_{i=0}^{d-1}
\LD\cal{P}\RD_{1,1,i}^{\mathbb{P}^1_{3,3,3}}\LD\Delta_{2,2;1;1}\RD_{0,4,d+1-i}^{\mathbb{P}^1_{3,3,3}}
\end{equation}
Considering other strata in Getzler's Relation, the integration over $\delta_{2,2},\delta_{2,3}$ and $\delta_{2,4}$ will all vanish for the following reasons:
\begin{itemize}
\item For $\delta_{2,2}$, $\LD\alpha,\beta,1\RD_{0,3,j}^{\mathbb{P}^1_{3,3,3}}=0$ for all     $\{\alpha,\beta\}\subset\Delta_{2,2;1;1}$.
\item For $\delta_{2,3}$, $\LD\alpha\RD_{1,1,j}^{\mathbb{P}^1_{3,3,3}}=0$ for all $\alpha\in\Delta_{2,2;1;1}$.
\item For $\delta_{2,4}$, $\LD 1,\alpha,\beta,-\RD_{0,4,j}^{\mathbb{P}^1_{3,3,3}}=0$ for all $\{\alpha,\beta\}\subset\Delta_{2,2;1;1}$.
\end{itemize}
As the integration of $\Lambda_{1,4,d+1}^{\mathbb{P}^1_{3,3,3}}(\Delta_{2,2;1;1})$ over $\delta_{0,3},\delta_{0,4},\delta_{\beta}$ only give genus-0 invariants, the Getzler's relation implies $\LD\cal{P}\RD_{1,1,d}^{\mathbb{P}^1_{3,3,3}}$ can be reconstructed from $\LD\cal{P}\RD_{1,1,d'}^{\mathbb{P}^1_{3,3,3}}$ with $d'<d$ and genus-0 primary correlators.

{\bf $\mathbb{P}_{4,4,2}^1$-case:}
Now we choose four insertions $\Delta_{x,3},\Delta_{x,2},\Delta_{y,1},\Delta_{z,1}\in\,H^*_{CR}(\mathbb{P}^1_{4,4,2})$ and denote by $\Delta_{3,2;1;1}$.
In this case, we use genus-0 computation:
$$\LD\Delta_{3,2;1;1}\RD_{0,4,0}^{\mathbb{P}^1_{4,4,2}}=0, \LD\Delta_{3,2;1;1}\RD_{0,4,1}^{\mathbb{P}^1_{4,4,2}}=\frac{1}{4}.$$
Integrating $\Lambda_{1,4,d+1}^{\mathbb{P}^1_{4,4,2}}(\Delta_{3,2;1;1})$ on the $\delta_{3,4}$, we have
\begin{equation*}
\int_{\delta_{3,4}}\Lambda_{1,4,d+1}^{\mathbb{P}^1_{4,4,2}}(\Delta_{3,2;1;1})
=\LD\cal{P}\RD_{1,1,d}^{\mathbb{P}^1_{4,4,2}}+\sum_{i=0}^{d-1}
\LD\cal{P}\RD_{1,1,i}^{\mathbb{P}^1_{4,4,2}}\LD\Delta_{3,2;1;1}\RD_{0,4,d+1-i}^{\mathbb{P}^1_{4,4,2}}
\end{equation*}
Again, integrations over $\delta_{2,2},\delta_{2,3},\delta_{2,4}$ are zero and over $\delta_{0,3},\delta_{0,4},\delta_{\beta}$ only give genus-0 contribution. Thus Getzler's Relation implies $\LD\cal{P}\RD_{1,1,d}^{\mathbb{P}^1_{4,4,2}}$ is reconstructable.

{\bf $\mathbb{P}_{6,3,2}^1$-case:}
Now we choose four insertions $\Delta_{x,5},\Delta_{x,2},\Delta_{y,1},\Delta_{z,1}\in\,H^*_{CR}(\mathbb{P}^1_{6,3,2})$ and denote them as $\Delta_{5,2;1;1}$.
In this case, we use genus-0 computation:
$$\LD\Delta_{5,2;1;1}\RD_{0,4,0}^{\mathbb{P}^1_{6,3,2}}=0, \LD\Delta_{5,2;1;1}\RD_{0,4,1}^{\mathbb{P}^1_{6,3,2}}=\frac{1}{6}.$$
Now we integrate $\Lambda_{1,4,d+1}^{\mathbb{P}^1_{6,3,2}}(\Delta_{5,2;1;1})$ over $\delta_{3,4}$,
\begin{equation*}
\int_{\delta_{3,4}}\Lambda_{1,4,d+1}^{\mathbb{P}^1_{6,3,2}}(\Delta_{5,2;1;1})
=\frac{2}{3}\LD\cal{P}\RD_{1,1,d}^{\mathbb{P}^1_{6,3,2}}+4\sum_{i=0}^{d-1}
\LD\cal{P}\RD_{1,1,i}^{\mathbb{P}^1_{6,3,2}}\LD\Delta_{5,2;1;1}\RD_{0,4,d+1-i}^{\mathbb{P}^1_{6,3,2}}
\end{equation*}
Other strata only give genus-0 correlators. Thus $\LD\cal{P}\RD_{1,1,d}^{\mathbb{P}^1_{6,3,2}}$ is reconstructable.
\end{proof}

\subsubsection{Higher genus reconstruction}

Higher genus reconstruction is identical to that of FJRW-theory.
By applying Lemma \ref{lm:g-reduction} to Gromov-Witten theory, we obtain
\begin{thm}
For elliptic orbifold $\mathbb{P}^1$, the ancestor potential function is uniquely determined by the genus-0 potential and the genus-1 primary potential.
\end{thm}
\begin{proof}
We consider the Gromov-Witten invariants for the elliptic orbifold $\mathbb{P}^1$,
\begin{equation*}
\LD\tau_{l_1}(\alpha_1),\cdots,\tau_{l_n}(\alpha_n),T_{i_1},\cdots,T_{i_k}\RD_{g,n+k,d}
=\int_{\overline{\mathscr{M}}_{g,n+k}}\Psi_{l_1,\cdots,l_n}\cdot\Lambda_{g,n+k,d}^{\mathscr{X}}(\alpha_1,\cdots,\alpha_n,T_{i_1},\cdots,T_{i_k})
\end{equation*}
where $\Psi_{l_1,\cdots,l_n}=\prod_i\bar\psi_i^{l_i}$.
The correlator will vanish except for
\begin{equation}\label{eq:degree g-red}
\deg\Psi_{l_1,\cdots,l_n}+\sum_{i=1}^{n}\deg\alpha_i+\sum_{j=1}^{k}\deg(T_{i_j})=2g-2+n+k.
\end{equation}

As long as $\deg\alpha_i\leq 1$ and $\deg T_{i_j}\leq 1$, we have $\deg\Psi_{l_1,\cdots,l_n}\geq2g-2$. Now we apply Lemma \ref{lm:g-reduction}. If $\deg\Psi_{l_1,\cdots,l_n}$ is large, then the integral is changed to the integral over the boundary classes while decreasing the degree of the total $\psi$-classes or $\kappa$-classes. After applying the splitting and composition laws, the genus involved will also
decrease. We can continue this process until the original integral
is represented by a linear combination of primary correlators of genus-0 and
genus-1.

Moreover, for primary genus-1 correlators, we have $\deg\Psi_{l_1,\cdots,l_n}=0$. Thus equation (\ref{eq:degree g-red}) holds if and only if
$\deg\alpha_i=\deg T_{i_j}=1$, i.e, we only need to consider genus-1 correlators of type $\LD \cal{P},\cdots,\cal{P}\RD_{1,n,d}$.
\end{proof}
\subsection{LG-to-CY Mirror theorem}
In this subsection, we establish the LG-to-CY mirror theorem for all genera between the Gromov-Witten theory and the Saito-Givental theory of $W_{\infty}$.

\begin{thm}\label{thm:LG-LG mirror}
When $W:=P_{8},X_{9},J_{10}$, we can choose the coordinates appropriately, such that the ancestor potential function
\begin{equation}
\cal{A}_{GW}^{W^T/\widetilde{G}_{\textrm{max}}}(\textbf{t}_A)=\cal{A}_{W}^{SG}(\textbf{t}_B).
\end{equation}
where $\textbf{t}_A, \textbf{t}_B$ are semi-simple points.
\end{thm}

First of all, according to \cite{MR}, the divisor axiom also holds for the ancestors of the singularities we consider in this paper.
Now we have the identical reconstruction theorem of Saito-Givental theory as that of Gromov-Witten-theory.
\begin{thm}
For the three types of elliptic singularitis $W=P_8, X_{9}, J_{10}$, the ancestor correlator $\LD\tau_{l_1}(\alpha_1),\cdots,\tau_{l_n}(\alpha_n)\RD_{g,n}^{W_{\infty}}$ in Saito-Givental theory is uniquely reconstructed by tautological relations (WDVV-relation, divisor axiom, Getzler relation and g-reduction) from the pairing, Milnor ring structural constants and genus-0 correlators
$\LD\,x,x,x\RD_{0}^{P_{8}(\infty)}$, $\LD\,x,y,z\RD_{0}^{P_{8}(\infty)}$, $\LD\,x,y,xz,xz\RD_{0}^{X_{9}(\infty)}$,
$\LD\,y,y,y\RD_{0}^{J_{10}(\infty)}$, $\LD\,x,x,xz\RD_{0}^{J_{10}(\infty)}$ and $\LD\,x,y,z\RD_{0}^{J_{10}(\infty)}$
respectively.
\end{thm}

To prove the LG-to-CY mirror theorem, we only need to match the above 3-point correlators. This is already computed by T. Milanov and Y. Ruan. For more details, we refer to \cite{MR}.

\subsubsection{Proof of LG-to-CY mirror theorem for the $\mathbb{P}^1_{3,3,3}$ case.}
\begin{proof}
The Picard-Fuchs equation in this case is
\begin{equation}\label{eq:Picard-Fuchs1}
\frac{d^2u}{d\sigma^2} + \frac{3\sigma^2}{\sigma^3+27} \frac{du}{d\sigma} +  \frac{\sigma}{\sigma^3+27} u = 0
\end{equation}
We set $\lambda=-\sigma^3/27$, the indicial equation is
\begin{equation}
\lambda^2-\frac{2}{3}\lambda+\frac{1}{9}=0
\end{equation}
It has double root $\lambda=1/3$. Near $\sigma=\infty$, the monodromy is  $\begin{pmatrix}1 & 0  \\ 3 & 1 \end{pmatrix}$. Thus we can choose an integral symplectic basis $A,B$ in $H^1(E_{\sigma},\mathbb{Z})$, such that $\pi_A(\sigma),\pi_{B}(\sigma)$ are solutions of equation (\ref{eq:Picard-Fuchs1}). Here we have
\begin{equation}\label{hgs}
\begin{split}
&\pi_{A}(\sigma)=-\sigma^{-1}\,_2F_1(\frac{1}{3},\frac{2}{3};1;-\frac{27}{\sigma^3})\\
&\pi_{B}(\sigma)=-\frac{3}{2\pi\,i}\pi_{A}(\sigma)\log\,(-\sigma)+\frac{3}{2\pi\,i}\,\sum_{k=1}^\infty\,b_k(-\sigma/3)^{-3k-1}.
\end{split}
\end{equation}
The coefficient in $\pi_{B}(\sigma)$ is determined by the monodromy around $\sigma=\infty$.
The parameter of the marginal deformation is $\tau=\frac{\pi_B(\sigma)}{\pi_A(\sigma)}$. The mirror map is defined by $q:=e^{2\pi i \tau/3}.$

We identify the correspondence between generators of the Chen-Ruan cohomology $H^*_{CR}(\mathbb{P}^1_{3,3,3})$ and the basis of the Frobenius manifold structure at $\sigma\longrightarrow\infty$:
\begin{eqnarray*}
\quad 1 &\mapsto& 1,\\
\cal{P}    &\mapsto& (\sigma^3+27)\pi_A^2(\sigma)xyz,\\
27^{\deg\Delta_i-\frac{1}{3}}\Delta_i       &\mapsto& (-1)^{\deg\phi_i-\frac{1}{2}} (\sigma^3+27)^{\deg\phi_i}\pi_A(\sigma) \phi_i(x),\quad 1\leq
i\leq 6.
\end{eqnarray*}
Here $\Delta_1=\Delta_{x,1}$, $\Delta_2=\Delta_{y,1}$, $\Delta_3=\Delta_{z,1}$.
Recall that the residue pairing at $\sigma\longrightarrow\infty$ is
\begin{equation*}
{\rm res}_{\textbf{x}=0}\ \frac{xyz}{f_{x}f_{y}f_{z}}\,d^3x = \frac{1}{\sigma^3+27}
\end{equation*}
Thus the pairing matches up. On the other hand, we set
$$
\LD\Delta_{x,1},\Delta_{y,1},\Delta_{z,1}\RD_{0,3}^{\mathbb{P}^1_{3,3,3}}
:=\sum_{d=0}^{\infty}\LD\Delta_{x,1},\Delta_{y,1},\Delta_{z,1}\RD_{0,3,d}^{\mathbb{P}^1_{3,3,3}}q^d
$$
Then the mirror map between the three-point correlators is
\begin{eqnarray*}
\LD\Delta_{x,1},\Delta_{y,1},\Delta_{z,1}\RD_{0,3}^{\mathbb{P}^1_{3,3,3}} &\mapsto& -i\pi_{A}(\sigma)=q+o(q)\\
\LD\Delta_{x,1},\Delta_{x,1},\Delta_{x,1}\RD_{0,3}^{\mathbb{P}^1_{3,3,3}} &\mapsto& \frac{i\sigma}{3}\pi_A(\sigma)=\frac{1}{3}+o(q)
\end{eqnarray*}
The leading term of the Fourier expansion of the two correlators  are $q$ and $\frac{1}{3}$.
According to the reconstruction theorem \ref{thm:GW recons}, this implies the CY-LG mirror symmetry of all genera for $\mathbb{P}^1_{3,3,3}$.
\end{proof}

\subsubsection{Proof of LG-to-CY mirror theorem for the $\mathbb{P}^1_{4,4,2}$ case.}

\begin{proof}
According to \cite{MR}, the Picard-Fuchs equation is the same as before. For the mirror maps, we refer the readers to
\cite{MR}. Especially, in this case, for the marginal direction, $q:=e^{2\pi i\tau/4}$, $\cal{P}$ is mapped to $\frac{4}{9}(27+\sigma^3)xyz\pi_{A}^2(\sigma)$. Thus the mirror map gives the same pairing, and assigns the genus-0 3-point correlators as follows:
\begin{eqnarray*}
\LD\Delta_{x,1},\Delta_{y,1},\Delta_{z,1}\RD_{0,3} &\mapsto& -i(\sigma^2-4)x^2y^2\pi_{A}^3(\sigma)=q+2q^5+o(q^5)\\
\LD\Delta_{x,2},\Delta_{x,1},\Delta_{x,1}\RD_{0,3} &\mapsto& -i(\sigma^2-4)x^3\pi_A^3(\sigma)=\frac{1}{4}+q^4+o(q^4)
\end{eqnarray*}
According to the reconstruction theorem \ref{thm:GW recons}, by comparing the leading term of the expansion of $q$, the LG-to-CY mirror symmetry holds for $\mathbb{P}^1_{4,4,2}$.
\end{proof}

\subsubsection{Proof of LG-to-CY mirror theorem for the $\mathbb{P}^1_{6,3,2}$ case.}
\begin{proof}
The Picard-Fuchs equation is again the same.
Applying the map introduced in \cite{MR}. For the marginal direction, $q=e^{2\pi i\tau/6}$, and $\cal{P}$ is mapped to $\frac{2}{3}(27+\sigma^3)xyz\pi_{A}^2(\sigma)$. It is easy to check that the ring structural constants match and
\begin{align*}
&\LD\Delta_{x,1},\Delta_{y,1},\Delta_{z,1}\RD_{0,3}^{\mathbb{P}^1_{6,3,2}}  &\mapsto\,q+2q^7+o(q^7)\\
&\LD\Delta_{y,2},\Delta_{y,1},\Delta_{y,1}\RD_{0,3}^{\mathbb{P}^1_{6,3,2}}  &\mapsto\frac{1}{3}+2q^6+o(q^6)\\
&\LD\Delta_{x,4},\Delta_{x,1},\Delta_{x,1}\RD_{0,3}^{\mathbb{P}^1_{6,3,2}}  &\mapsto\frac{1}{6}+q^6+o(q^6)\\
\end{align*}
By comparing the leading term of the expansion of $q$, the LG-to-CY mirror symmetry holds for $\mathbb{P}^1_{6,3,2}$.
\end{proof}
\subsection{Convergence of Gromov-Witten theory}

Recall $\mathscr{B}$ is the set of \emph{canonical generators} of $H^*_{CR}(\mathscr{X})$ introduced in Section 4.1. We define
\begin{equation}
I_{g,n,d}^{GW}:=\max_{\alpha_i\in\mathscr{B}}\Big|\LD\alpha_1,\cdots,\alpha_n\RD_{g,n,d}\Big|
\end{equation}
We assume that for $l_i=0$, $\alpha_i\neq\cal{P}$, thus we can define
\begin{equation}
I_{g,n,k,d,L}^{GW}:=\max_{\begin{subarray}{L}
l=\sum_{i}l_i\\\alpha_i\in\mathscr{B}
\end{subarray}
}\Big|\LD\tau_{l_1}(\alpha_1),\cdots,\tau_{l_n}(\alpha_n),\cal{P},\cdots,\cal{P}\RD_{g,n+k,d}\Big|
\end{equation}
Now let us define the \emph{length} of a genus-0 Gromov-Witten correlator.
\begin{df}
We say the correlator is of \emph{length} $0$ if it contains an insertion $\cal{P}$. We say it is of \emph{length} $m $ $(m\geq1)$ if after applying at most $m$ times WDVV equation, it can be reconstructed from linear products of length $0$ correlators, or genus-0 correlators with fewer marked points or lower degree.
\end{df}
Let $I_{0,n+3,d}^{GW}(m)$ be the maximum absolute value of all genus-0 $(n+3)$-points correlators with degree $d$ and length $m$. Let $I_{0,n+3,d}^{GW}(-1)$ be the maximum absolute value of $I_1+I_2+I_3$ in the WDVV equation (\ref{eq:WDVV-GW}). Thus
\begin{equation}\label{eq:length-GW}
I_{0,n+3,d}^{GW}(m)\leq3I_{0,n+3,d}^{GW}(m-1)+I_{0,n+3,d}^{GW}(-1), m\geq1
\end{equation}
By carefully using the reconstruction, we have the following estimation,
\begin{lm}\label{lm:basic}
We have
$$I_{0,3,0}^{GW}\leq1, I_{0,4,0}^{GW}\leq1/4$$
For $d\neq0$, we have:
\begin{eqnarray*}
\begin{split}
&I_{0,3,d}^{GW}\leq\,d^{-1}C^{d-1}\\
&I_{0,4,d}^{GW}\leq\,C_0\,d^{-1}C^{d-1}
\end{split}
\end{eqnarray*}
where $C_0$ is a number depends on $g$, $C_0^2\ll C$.
\end{lm}
We give the proof in the appendix. Now we can continue on with more insertions,
\begin{lm}
For $n+d\geq4,$ we have:
\begin{equation*}
I_{0,n,d}^{GW}\leq
\left\{
\begin{array}{ll}
d^{n-5}C^{n+d-4}, &d\geq1. \\
C^{n-4},&d=0.
\end{array}\right.
\end{equation*}
\end{lm}

\begin{proof}
Lemma \ref{lm:basic} implies the estimation holds for $n\leq4$. We assume it holds for $\leq\,n+2$, where $n\geq2$. Now we prove it for $n+3$.
We recall the terms in the WDVV equation (\ref{eq:WDVV-GW}).
For $n\geq2, d=0$, we have
\begin{equation}
I_{0,n+3,0}^{GW}\leq\sum_{i=1}^{n-1}C_0 {n\choose i}\,I_{0,3+i,0}^{GW}I_{0,n+3-i,0}^{GW}\leq\,C_0\,2^{n-1}\,C^{n-2}\leq\,C^{n-1}
\end{equation}
For $d\geq1$, we calculate $I_{0,n+3,d}^{GW}(-1)$ and $I_{0,n+3,d}^{GW}(0)$ first. The divisor equation implies
\begin{equation}
I_{0,n+3,d}^{GW}(0)
\leq\,d\,I_{0,n+2,d}^{GW}
\leq\,d^{n-2}\,C^{n+d-2}
\end{equation}
Next, we have
\begin{eqnarray*}\begin{split}
\Big|I_1\Big|
&\leq72\sum_{j=1}^{n-1}{n\choose j}I_{0,j+3,d}^{GW}I_{0,n+3-j,0}^{GW}\leq72\sum_{j=1}^{n-1}{n\choose j}d^{j-2}C^{n+d-2}\leq72*2^{n}d^{n-2}C^{n+d-2}\\
\Big|I_2\Big|
&\leq36\sum_{i=1}^{d-1}\sum_{j=0}^{n}{n\choose j}I_{0,j+3,d-i}^{GW}I_{0,n+3-j,i}^{GW}
\leq36\sum_{i=1}^{d-1}\sum_{j=0}^n{n\choose j}(d-i)^{j-2}i^{n-j-2}C^{n+d-2}\\
&\leq288d^{n-2}C^{n+d-2}\\
\Big|I_3\Big|
&\leq72I_{0,3,d}^{GW}I_{0,n+3,0}^{GW}\leq72d^{-2}C^{n+d-2}
\end{split}\end{eqnarray*}
For the estimation of $I_2$, we use for any $1\leq\,i\leq\,d$,
$$
\sum_{j=1}^{n-1}{n\choose j}(\frac{i}{d})^{j}(\frac{d-i}{d})^{n-j}\leq1
$$
and
\begin{equation}\label{eq:minus-square}
\sum_{i=1}^{d-1}i^{-2}(d-i)^{-2}\leq6\,d^{-2}
\end{equation}
Now we have,
$$
I_{0,n+3,d}^{GW}(-1)\leq\Big|I_1+I_2+I_3\Big|\leq18\frac{2^{n+2}+16+4}{C}d^{n-2}C^{n+d-1}
$$
It is easy to see the maximum length is less than $12$, by using (\ref{eq:length-GW}) repeatedly, we know
\begin{eqnarray*}
\begin{split}
I_{0,n+3,d}^{GW}&\leq3^{12}I_{0,n+3,d}^{GW}(0)+\frac{3^{12}-1}{2}I_{0,n+3,d}^{GW}(-1)
\leq d^{n-2}C^{n+d-1}
\end{split}
\end{eqnarray*}

\end{proof}

\begin{lm}
For genus-1 primary correlators, for $n\geq1$, we have
\begin{equation*}
I_{1,n,d}^{GW}\leq\left\{
\begin{array}{ll}
d^{2n-3}\,C^{n+2d-2},&
  \text{if} \,d\neq0. \\
1, & \text{if} \,d=0.
\end{array}
\right.
\end{equation*}
\end{lm}
\begin{proof}
The nonvanishing terms are just $\LD \cal{P},\cdots, \cal{P}\RD_{1,n,d}^{\mathcal{X}}$. For $d=0$, it is easy. For $d\neq0$, the divisor axiom implies we only need to verify for $n=1$.
We prove the estimation for the $\mathscr{X}=\mathbb{P}^{1}_{3,3,3}$ case. The other two cases is similar. When we integrate $\Lambda_{1,4,d+1}^{\mathbb{P}^1_{3,3,3}}(\Delta_{2,2;1;1})$ over Getzler's relation, the genus-1 contribution only comes from $\delta_{3,4}$.
Recall equation (\ref{eq:genus-1}), for $d\geq3,$
\begin{eqnarray*}
\begin{split}
\Big|\sum_{i=0}^{d-1}\LD\cal{P}\RD_{1,1,i}\LD\Delta_{2,2;1;1}\RD_{0,4,d+1-i}\Big|
&\leq\sum_{i=1}^{d-1}i^{-1}C^{2i-1}(d+1-i)^{-1}C^{d+1-i}+(24d)^{-1}C^{d+1}\\
&\leq\frac{1}{6d}C^{2d-1}
\end{split}
\end{eqnarray*}
We also have
\begin{eqnarray*}
\begin{split}
\Big|\int_{\delta_{\beta}}\Lambda_{1,4,d+1}^{\mathbb{P}^1_{3,3,3}}(\Delta_{2,2;1;1})\Big|
&\leq\sum_{i=0}^{d+1}I_{0,4,i}^{GW}I_{0,4,d+1-i}^{GW}\leq\sum_{i=1}^{d+1}i^{-1}(d+1-i)^{-1}C^{2d}+2(d+1)^{-1}C^{d+1}\\
&\leq\frac{1}{d}C^{2d+1}
\end{split}
\end{eqnarray*}
Other genus-0 contributions are all bounded by $d^{-1}C^{2d+1}$. Thus the estimation follows from the Getzler relation.
\end{proof}

Here is the main estimation in this section,
\begin{thm}
Let us denote $\chi:=2g-2+n$. Then for $\chi\geq0$, we have:
\begin{equation}\label{eq:g-red GW}
I_{g,n,0,d,L}^{GW}\leq\left\{
\begin{array}{ll}
d^{\chi-2}C(g,n)^{\chi+(g+L+1)d-2}, &
  \text{if} \,d\neq0. \\
C(g,n)^{\chi-1},& \text{if} \,d=0.
\end{array}
\right.
\end{equation}
Here $C(g,n)$ is a sufficient large number which depends only on $g$ and $n$, satisfying:
\begin{equation*}
\left\{\begin{array}{ll}
C(0,3)=C(0,4)>>0&\\
C(g,n+1)>C(g,n)& \text{except\,for\,} g=0, n\leq3\\
C(g+1,n)>C(g,n)+1&
\end{array}\right.
\end{equation*}
\end{thm}
\begin{proof}
Recall the expression of the ancestors,
\begin{equation*}
\LD\tau_{l_1}(\alpha_1),\cdots,\tau_{l_n}(\alpha_n)\RD_{g,n,d}
=\int_{\overline{\mathscr{M}}_{g,n}}\Psi_{l_1,\cdots,l_n}\cdot\Lambda_{g,n,d}^{\mathscr{X}}(\alpha_1,\cdots,\alpha_n)
\end{equation*}
where $\Psi_{l_1,\cdots,l_n}=\prod_i\bar\psi_i^{l_i}$. We denote by $\deg\Psi_{l_1,\cdots,l_n}=\sum_{i}l_i:=L$.
Now we use $g$-reduction. According to degree counting for elliptic orbifold $\mathbb{P}^1$, the integration can be expressed as linear combinations of primary correlators with genus 0 or 1,
\begin{equation*}
\int_{\overline{\mathscr{M}}_{g,n}}\Psi_{l_1,\cdots,l_n}\cdot\Lambda_{g,n,d}^{\mathscr{X}}(\alpha_1,\cdots,\alpha_n)
=\sum_{j=1}^{N}\int_{\Gamma_j}\Lambda_{g,n,d}^{\mathscr{X}}(\alpha_1,\cdots,\alpha_n)
\end{equation*}
where $\Gamma_j$ is a dual graph with $v(\Gamma_j)=L$. Each component of $\Gamma_j$ is either of genus-0 or genus-1. $N$ is the number of such dual graphs, depends only on genus $g$ and number of marked points $n$.
For simplification, we always denote $C$ as the $C(g,n)$ for the corresponding correlator we want to estimate.

If $L=0$, the result follows from the previous three lemmas. For $L\geq1$, $d=0$,
\begin{eqnarray*}
\begin{split}
\Big|\int_{\Gamma_j}\Lambda_{g,n,0}^{\mathscr{X}}(\alpha_1,\cdots,\alpha_n)\Big|
\leq\,6^{L}\prod_{i=0}^{L}I_{g_i,n_i+k_i,0}^{GW}
\leq\,6^{L}\prod_{i=0}^{L}C^{\chi_i-1}
\leq\,6^{L}C^{\chi-1-L}
\end{split}
\end{eqnarray*}
where $6^L$ is the upper bound for the products of pairing factors $\eta^{\mu,\nu}$ from $L$ nodes. Thus we have
\begin{equation*}
I_{g,n,0,0,L}^{GW}\leq\frac{6^L\,N}{C^L}C^{\chi-1}\leq\,C^{\chi-1}
\end{equation*}

For $d\geq1$, we consider for each dual graph $\Gamma_j$, with $k_j$ the number of nodes for $\Gamma_j$. Thus we know $k_j\geq1, \sum_{i}k_j=2L$.
Here we need to deal with those terms with $d_i=0$, as the exponent of $C(g,n)$ in the estimation will increase by $1$ in these cases.
We generalize (\ref{eq:minus-square}) and will have the following inequality:
\begin{equation*}
\sum_{d_0+\cdots+d_L=d}\prod_{d_i\neq0}d_i^{\chi_i-2}\leq8^{L}d^{\chi-2}.
\end{equation*}
Let $K$ be the number of $d_i$ which is $0$, then $K\leq\,L$. Now we have the estimation,
\begin{eqnarray*}
\begin{split}
\Big|\int_{\Gamma_j}\Lambda_{g,n,d}^{\mathscr{X}}(\alpha_1,\cdots,\alpha_n)\Big|
\leq&\sum_{d_0+\cdots+d_L=d}6^{L}\prod_{i=0}^{L}I_{g_i,n_i+k_i,d_i}^{GW}\\
\leq&\sum_{d_0+\cdots+d_L=d}6^{L}C^{K}\prod_{d_i\neq0}d_i^{\chi_i-2}\prod_{i=0}^{L}C^{\chi_i+(g_i+1)d_i-2}\\
\leq&6^{L}8^{L}d^{\chi-2}C^{\chi+(g+1)d-2}
\end{split}
\end{eqnarray*}
Now (\ref{eq:g-red GW}) follows from $L\cdot\,d\neq0$ in this case.
\end{proof}

Now we prove the convergence of the Gromov-Witten part in Theorem \ref{thm:convergent}.
\begin{proof}
For $\textbf{t}=s\cal{P}$, recall the Divisor equation for the ancestor correlators (\ref{eq:divisor}),
we have the estimation:
\begin{eqnarray*}
\begin{split}
&\Big|\LD\LD\tau_{l_1}(\alpha_1),\cdots,\tau_{l_n}(\alpha_n)\RD\RD_{g,n}(s\cal{P})\Big|\\
&=\Big|\sum_{d\geq0}\sum_{k\geq0}\frac{1}{k!}\LD\tau_{l_1}(\alpha_1),\cdots,\tau_{l_n}(\alpha_n),s\cal{P},\cdots,s\cal{P}\RD_{g,n+k,d}q^d\Big|\\
&\leq\Big|\LD\tau_{l_1}(\alpha_1),\cdots,\tau_{l_n}(\alpha_n)\RD_{g,n,0}\Big|+\Big|\sum_{d\geq1}\sum_{k\geq0}\frac{1}{k!}\LD\tau_{l_1}(\alpha_1),\cdots,\tau_{l_n}(\alpha_n)\RD_{g,n,d}q^ds^kd^k\Big|\\
&\leq\sum_{d\geq0}\Big|q\,e^s\Big|^dd^{\chi-2}C(g,n)^{\chi+(g+L+1)d-2}
\end{split}
\end{eqnarray*}
This is convergent for $\Big|q\,e^s\,C(g,n)^{g+L+1}\Big|\leq1/2$.

Now we consider $\textbf{t}=s\cal{P}+\sum_{i\geq0}t_i\phi_i$, where $\phi_i$ ranges over the homogeneous basis other than $\cal{P}$. The dimension formula (\ref{form:GW dim}) implies the ancestor function $\LD\LD\tau_{l_1}(\alpha_1),\cdots,\tau_{l_n}(\alpha_n)\RD\RD_{g,n}(\textbf{t})$ is a polynomial of $t_i$ with coefficients are ancestor functions valued at $s\cal{P}$. As the number of terms of the monomials in this polynomial depends only on the genus $g$ and number of marked points $n$. It follows that for $g,n$ fixed, $\LD\LD\tau_{l_1}(\alpha_1),\cdots,\tau_{l_n}(\alpha_n)\RD\RD_{g,n}(\textbf{t})$ is convergent.
\end{proof}

\appendix
\section{Proof of Lemma \ref{lm:key-GW}}
\begin{proof}
We already gave the reconstruction for Type 1 and Type 6 correlators. Here are the proofs for other types case by case by using (\ref{eq:WDVV-GW}).
\begin{itemize}
\item For Type 2, if $d>0$, we consider $\LD\gamma',\gamma',\gamma,\gamma,\beta'\star\beta\RD_{0,5,d}$, $\beta\neq\gamma$. As $\deg\gamma\leq\deg\beta$, we have $\gamma'\star\gamma'=\gamma'\star\beta'=\gamma'\star\beta=0$.
Lemma \ref{lm:vanish n=4} implies $I_3$ also vanish. The reconstruction follows by
\begin{equation}\label{eq:Type2}
\LD\gamma',\gamma',\gamma,\gamma\RD_{0,4,d}
=\frac{|\beta|}{d}\LD\gamma',\gamma',\gamma,\gamma,\beta'\star\beta\RD_{0,5,d}
=\frac{|\beta|}{d}(I_1+I_2)
\end{equation}

\item
For Type 3 case, for $d>1$, we consider $\LD\alpha,\beta,\gamma,\gamma\star\gamma'\RD_{0,4,d}$, where $\gamma$ has the greatest order among all primitive elements. In this case, $I_0$ and $I_1$ both vanish. Thus Lemma \ref{lm:vanish n=4} implies $$I_3=-\LD\alpha,\beta,\gamma\RD_{0,3,d}\eta^{\gamma,\gamma'}\LD\gamma',\gamma,\gamma,\gamma'\RD_{0,4,0}
=|\gamma|^{-1}\LD\alpha,\beta,\gamma\RD_{0,3,d}$$
Thus we have
\begin{equation}\label{eq:Type3}
\LD\alpha,\beta,\gamma\RD_{0,3,d}=\frac{|\gamma|}{(d-1)}I_2
\end{equation}

\item For Type 4 case, we can first reduce to the case of $\LD\alpha,\beta,\beta^i\RD_{0,3,d}$, $d>0$. Now choose $\gamma$ the rest primitive element and apply (\ref{eq:WDVV-GW}) to $\LD\alpha,\beta^i,\beta,\gamma'\star\gamma\RD_{0,4,d}$. Then $I_0,I_1$ and $I_3$ all vanish. Thus the reconstruction follows by
\begin{equation}\label{eq:Type4}
\LD\alpha,\beta,\beta^i\RD_{0,3,d}=\frac{|\gamma|}{d}I_2
\end{equation}
\item
For Type 5, for $d>0$, by induction, we already know $\LD\alpha,\gamma,\beta\RD_{0,3,d+1}$ and Type 1 correlators with degree $d$ are recontructable.
Now we apply (\ref{eq:WDVV-GW}) to $\LD\alpha,\gamma,\beta,\beta\star\beta'\RD_{0,4,d}$. Except for $\LD\alpha,\gamma,\beta\RD_{0,3,1}\eta^{\beta,\beta'}\LD\beta',\beta,\beta,\beta'\RD_{0,4,d}$, we already know all the terms in the equality are reconstructable. This gives the recursion for $\LD\beta',\beta,\beta,\beta'\RD_{0,4,d}$.
\end{itemize}
\end{proof}
\section{Proof of Lemma \ref{lm:basic}}
\begin{proof}
The key idea of the estimation is use induction. As we already have the recursion formulas from Lemma \ref{lm:key-GW}. For each type of correlators, we induct on the degree $d$. It is easy to see the lemma holds true for $d\leq1$.
The computation is tedious but elementary. Here we only give the estimation for Type 1 and Type 6 correlators. Others are similarly obtained.
\begin{itemize}
\item For Type 1 correlators, we apply (\ref{eq:WDVV-GW}). Assume $M$ is the maximum length for Type 1 genus-0 4-point correlators. Thus inequality (\ref{eq:length-GW}) implies
\begin{eqnarray*}
\begin{split}
I_{0,4,d}&\leq3^{M}I_{0,4,d}(0)+\frac{3^{M}-1}{2}I_2\\
&\leq3^{M}d^{-1}C^{d-1}+(3^{M}-1)36\sum_{i=1}^{d-1}(d-i)^{-2}i^{-1}C^{d-2}\\
&\leq\,C_0 d^{-1}C^{d-1}\end{split}
\end{eqnarray*}

\item For Type 6 correlators, the equation (\ref{eq:Type6}) implies
\begin{equation*}
\Big|\LD\alpha^{j},\alpha,\alpha\RD_{0,3,d}\Big|
\leq\frac{72|\beta|}{d^2}\sum_{i=1}^{d-1}\frac{1}{d-i}C^{d-i-1}\frac{C_0}{i}C^{i-1}
\leq\frac{864C_0}{C\,d^2}C^{d-1}
\leq\,d^{-2}C^{d-1}
\end{equation*}




\end{itemize}

\end{proof}

\end{document}